\documentclass[reqno,11pt, letterpaper]{amsart}

\PassOptionsToPackage{fleqn}{amsmath}
\RequirePackage{amsthm}
\RequirePackage{amsmath}
\RequirePackage{latexsym}
\RequirePackage{amssymb}
\RequirePackage{amscd}
\RequirePackage{epsfig}
\RequirePackage{ifthen}
\RequirePackage{varioref}
\usepackage{graphicx}
\usepackage{changes}
\usepackage{soul}
\usepackage{xcolor}
\usepackage{color}
\usepackage[misc]{ifsym}
\usepackage{bm}
\usepackage{lipsum}
\usepackage{subcaption}
\usepackage{enumerate}
\usepackage{enumerate}
\usepackage{setspace}
\usepackage{bookmark}
\usepackage{hyperref}
\usepackage{tikz}

\usepackage{algorithm}
\usepackage{algorithmic}
\usepackage{mathrsfs}
\usepackage{amsthm}
\usepackage{mathrsfs}

\usepackage{epstopdf}

\hypersetup{hidelinks,
	colorlinks=true,
	allcolors=blue,
	pdfstartview=Fit,
	breaklinks=true}

\def\R{{\mathbb R}}

\def\E{{\mathbb E}}

\theoremstyle{plain}
\newtheorem{thm}{Theorem}[section]
\newtheorem{lem}[thm]{Lemma}
\newtheorem{prop}[thm]{Proposition}
\newtheorem{cor}[thm]{Corollary}
\theoremstyle{definition}
\newtheorem{rem}[thm]{Remark}

\newtheorem{exa}[thm]{Example}

\def\eq#1{{\rm(\ref{#1})}}

\def\Bbb#1{{\mathbb#1}}

\def\R{\Bbb R}

\numberwithin{equation}{section}
\newcommand\blfootnote[1]{%
 \begingroup
 \renewcommand\thefootnote{}\footnote{#1}%
 \addtocounter{footnote}{-1}%
 \endgroup
 }
\allowdisplaybreaks[1]

\makeatletter
\@namedef{subjclassname@2020}{%
  \textup{2020} Mathematics Subject Classification}
\makeatother

\begin{document}

\title[Geometric flows and space-periodic solitons on the light-cone]{Geometric flows and space-periodic solitons \\ on the light-cone}

    \author[Y. Yang]{Yun Yang${}^*$}\blfootnote{${}^*$~Corresponding author: yangyun@mail.neu.edu.cn}
    \address{ Yun Yang\newline\indent
     Department of Mathematics, Northeastern University, Shenyang, 110819, P.R. China}
    \email{yangyun@mail.neu.edu.cn}

\begin{abstract}
This paper investigates curve flows on the light-cone in the 3-dimensional Minkowski space. We derive the Harnack inequality for the heat flow and present a detailed classification of space-periodic solitons for a third-order curvature flow. The nontrivial periodic solutions to this flow are expressed in terms of the Jacobi elliptic sine function. Additionally, the closed soliton solutions form a family of transcendental curves, denoted by $\mathrm{x}_{p,q}$, which are characterized by a rotation index $p$ and close after $q$ periods of their curvature functions. The ratio $p/q$ satisfies $p/q \in (\sqrt{2/3}, 1)$, where $p$ and $q$ are relatively prime positive integers. Guided by the classification process, we obtain the analytic solutions to a second-order nonlinear ordinary differential equation.
\end{abstract}




\maketitle
\section{Introduction}
The light-cone represents the boundary in which information can travel, with the speed of light serving as the ultimate limit, and
the idea arises directly from Einstein's theory of special relativity \cite{ein}, which states that the speed of light is constant for all observers, regardless of their motion.
Originating from a specific event at its vertex, the light-cone divides space-time into future and past cones, illustrating the spatial and causal relationships between events \cite{mtw, nak}. In mathematics, Lorentzian manifolds \cite{minkowski1908,vr1} provide a framework for modeling curved space-time, with the light-cone characterized using semi-Riemannian geometry to examine the metric and curvature properties of space-time \cite{nei, sw}.
In fact, the light-cone is a pivotal instrument in comprehending degenerate geometry of space-time manifolds \cite{duggal1996}. It can be envisioned as the light-speed hyperplane at a particular space-time point, profoundly influencing the overall metric structure of space-time. Furthermore, it serves as the foundation for defining geodesics \cite{bf1994}, which delineate the paths along which particles traverse the shortest conceivable routes within a given space-time geometry. Additionally, the light-cone is intricately linked to gravitational fields and the curvature of space-time, as elucidated by Einstein's field equations \cite{ps}.
In physics, the light-cone casts a long shadow, defining the ultimate velocity of information dissemination--the speed of light, a cornerstone of special relativity \cite{rin}.
Through the mathematical framework of light-cones, Hawking and Ellis systematically explores causal structures, singularities, black holes, and cosmology \cite{he}.
The light-cone has been instrumental in proving singularity theorems, analyzing black hole horizons, and gaining insights into the universe's expansion and structure  \cite{hawking1965,hawking1966a,hawking1966b,hawking1967,pen,penrose2020}.

Periodic solutions represent a vital mathematical tool in describing periodic phenomena and unraveling the behavior of nonlinear systems.
In mathematical physics, the study of periodic solutions not only elucidates periodic occurrences in nature but also lays the theoretical groundwork for comprehending nonlinear physical systems \cite{bir,kov,poi}.
These solutions are ubiquitous in fields like mechanics, vibration theory, optics, electromagnetism, chaos, and bifurcations (for more details see \cite{str}).
While exact solutions for numerous physical systems remain elusive, periodic solutions can be procured through numerical methods or approximations,
thereby aiding in the anticipation of system behavior and guiding experimental endeavors \cite{ptvf}.

Geometric flows are generally described by partial differential equations that govern the evolution of geometric structures, such as surfaces, manifolds, or other geometric objects (for more details refer to \cite{brakke1978, evans1991, hamilton1982, olv-isf}). These equations define how the geometric object evolves over time, shaping its properties in response to the flow.
A key concept in the study of geometric flows is that of soliton solutions, which arise from nonlinear partial differential equations. Solitons are stable solutions, particularly in the context of wave propagation. A distinctive property of solitons is that they maintain their shape during propagation \cite{gardner1967, zabusky1965}. This feature has found analogous applications in geometric flows, where solitons often serve as steady solutions that reflect the stability of certain geometric structures under the flow (refer to \cite{andrews2004, hs, perelman2002}).
In the study of geometric flows, one of the most crucial challenges is understanding the nature of singularities. These singularities are intimately connected to the topological structure of the manifold. Investigating them is essential for gaining a deeper understanding of the manifold's geometric and topological properties. Moreover, such investigations are critical for revealing the manifold's multifaceted behavior. Soliton solutions play a significant role in this context, as they provide a prototypical framework for both the emergence and resolution of singularities in geometric flows (see \cite{hui1, hui2} for further details).

In this paper, we primarily focus on the curve flow occurring on the right half of a light-cone, denoted as  $LC^*$, which can be expressed in spherical coordinates as $\displaystyle (\psi, \psi\cos\theta, \psi\sin\theta)$  subject to the constraint $\psi>0$.
Let $\mathbf{r}_0$ represent a curve situated on $LC^*$. We then consider the motion of curves on  $LC^*$ characterized as
\begin{align}\label{lightflow}
\frac{\partial \mathbf{r}}{\partial t}=U\mathbf{r}+W\mathbf{T},\qquad \mathbf{r}(\cdot\;,0)=\mathbf{r}_0,
\end{align}
where  $\mathbf{T}$ denotes the unit tangent vector field. Notably,  $U$ and $W$ remain invariant under Lorentz transformations in the Minkowski space $\E^3_1$.
In Section \ref{sec-ev},  we demonstrate that this flow guarantees that the motion of the curves will consistently remain confined to $LC^*$.

This flow induces the evolutionary equations for the following quantities
\begin{equation*}
  \frac{g_t}{g}=U+W_s,
\end{equation*}
and
\begin{equation*}
  (k_g)_t=U_{ss}-2k_gU+W(k_g)_s,
\end{equation*}
where $g$ is the metric of curve $\mathbf{r}$, which depends on the pseudo-scalar product introduced in Section \ref{sec-back}. Here,  $s$ denotes the corresponding arc length parameter, and
 $k_g$ signifies the curvature of the curve $\mathbf{r}$ on $LC^*$ as defined in \eqref{FS-eq-light}.

Geometric heat flow \cite{hamilton1989} refers to the evolution of a family of submanifolds
    $F:M\times(0,T)\rightarrow N$ that satisfies the partial differential equation
    \begin{equation*}
       \frac{\partial}{\partial t}F=\Delta F,
    \end{equation*}
where $F(\cdot, 0)=F_0:M\rightarrow N$ represents the initial condition. Here, $\Delta$ denotes the Laplace-Beltrami operator on  $(M,g)$, with $g$  being the metric on the manifold $M$.
The curve shortening flow (CSF) is one of the simplest and well-studied models \cite{gage1983,gage1984,gagehamilton1986}.
The higher dimensional analogue of CSF is the mean curvature flow (MCF) \cite{brakke1978,coldingminicozzi2003,hui1}.
The Ricci flow (RF) \cite{hamilton1982, hamilton1995} deforms an initial metric in the direction of its Ricci tensor, and shares many similarities with CSF and MCF.
All these flows are collectively referred to as geometric heat flows. There are some analogical extensions of the geometric heat flow in affine geometry \cite{ast,oqy,st,yang2023maximal}.

When we choose $U=k_g$ and $W=0$, the flow given by equation \eqref{lightflow} simplifies to the heat flow, which is expressed as
\begin{equation}\label{curve-evo-light}
  \frac{\partial \mathbf{r}}{\partial t}=k_g\mathbf{r}.
\end{equation}
Consequently, the evolutionary equations undergo a modification, becoming $\displaystyle \frac{g_t}{g}=k_g$ and $\displaystyle (k_g)_t=(k_g)_{ss}-2k_g^2$.
To gain further insight into this flow, consider the following illustrative example.
\begin{exa}The curves defined by
\begin{equation*}
 \mathbf{r}=\left(\sqrt{-t}, \sqrt{-t}\cos\theta, \sqrt{-t}\sin\theta\right),
\end{equation*}
evolve according to the heat flow. During this evolution, we have
\begin{equation*}
 k_g=\frac{1}{2t},
\end{equation*}
where time $t$ ranges in $(-\infty, 0)$.
\end{exa}
In \cite{silva2023}, several classification results for solitons of this heat flow on the light-cone have been obtained.

The Li-Yau Harnack inequality, introduced in \cite{li1986}, serves as a potent instrument in the study of geometric flows. It offers profound insights into various aspects of these flows, including the behavior of solutions, their regularity, the formation of singularities, and the long-term convergence of the flow dynamics. This inequality has found widespread application in analyzing Ricci flow, mean curvature flow, and numerous other geometric flows, as demonstrated in \cite{andrews1994, bryan2020, chow1991, hamilton1993, hamilton1995h}.
By taking the notation $k=-k_g$, we derive a Harnack inequality specific to the heat flow on the light-cone  $LC^*$ (see Section \ref{subsec-hk} for the proof).
\begin{thm}[Harnack Inequality]\label{thm-hk}
For an immersed solution to the curve flow \eqref{curve-evo-light} on the light-cone $LC^*$ with $k>0$ and defined on the time interval $[t_0,T)$,  we have the following Harnack inequality
\begin{equation*}
  k_t-\frac{k_s^2}{k}-k^2+\frac{k}{2(t-t_0)}\geq0.
\end{equation*}
\end{thm}

Another noteworthy flow arises when we set $U=(k_g)_s$ and $W=-k_g$, that is,
\begin{equation}\label{soflow}
  \frac{\partial \mathbf{r}}{\partial t}=(k_g)_s\mathbf{r}-k_g\mathbf{T},
\end{equation}
leading to the evolutionary equations
\begin{equation*}
g_t=0,\qquad (k_g)_t=(k_g)_{sss}-3k_g(k_g)_s.
\end{equation*}
These equations embody a KdV (Korteweg-de Vries) equation,  a renowned nonlinear partial differential equation that is extensively utilized to model wave behaviors across various physical scenarios.
In this specific flow, the condition $g_t=0$ signifies that the flow exhibits non-stretching characteristics.
Consequently, we denominate this flow as the third-order centro-affine curvature flow.

In fact, when both  $g_t=0$ and $(k_g)_t=0$ are satisfied, it signifies that the curves retain their original shapes throughout the evolutionary process driven by the flow defined in \eqref{lightflow}.
Hence, these curves constitute a distinctive category of solitons for the flow. Through integration, we obtain
\begin{equation}\label{solitoneq}
   (k_g)_{ss}-\frac{3}{2}k_g^2+\frac{\lambda}{2}=0, \quad \left((k_g)_s\right)^2-k_g^3+\lambda k_g+\mu=0,
 \end{equation}
 where $\lambda$ and $\mu$ are the integration constants.
 Our primary focus is on studying nontrivial periodic soliton solutions. Notably, nontrivial periodic solutions to \eqref{solitoneq} may exist only when
 $\displaystyle \lambda>3\left(\mu/2\right)^{2/3}$.
 In this case, the cubic eqaution $x^3-\lambda x-\mu=0$ have three distinct real root, which we denote as  $x_1<x_2<x_3$.
 Furthermore, the periodic solutions for \eqref{solitoneq} can be expressed in terms of the Jacobi sine function
 \begin{equation*}
  k_g=x_1+(x_2-x_1)\mathrm{sn}^2\left(\frac{\sqrt{x_3-x_1}}{2}s,\;\sqrt{\frac{x_2-x_1}{x_3-x_1}}\right),
\end{equation*}
and the task to identify the periodic solitons is equivalent to solving the equation
\begin{equation}\label{secod-ode}
   \psi_{ss}-\frac{\psi_{s}^2+1}{2\psi}-\left(x_1+(x_2-x_1)\mathrm{sn}^2\left(\frac{\sqrt{x_3-x_1}}{2}s,\;\sqrt{\frac{x_2-x_1}{x_3-x_1}}\right)\right)\psi=0.
\end{equation}
For more detailed information, refer to Section \ref{sec-soliton}.
Equation \eqref{secod-ode} represents a complex second-order ordinary differential equation, and typically, finding its analytical solutions is a challenging task.
However, during our endeavor to classify periodic soliton solutions, we have been fortunate to uncover some analytical solutions to this equation.

In the Minkowski space $\E^3_1$, periodic curves can be created by rotating the points in the initial period around axes that can be either space-like, time-like, or light-like.
Our analysis of the periodic solitons related to \eqref{solitoneq} shows that the nature of their rotation hinges on the value of $\mu$.
Specifically,
\begin{itemize}
  \item When $\mu>0$, the periodic solitons rotate around time-like axis.
  \item When $\mu=0$, the periodic solitons rotate around light-like axis.
  \item When $\mu<0$, the periodic solitons rotate around space-like axis.
\end{itemize}
Importantly,  only when $\mu>0$ can closed solitons be formed .

For $\mu>0$,  an appropriate rotation can align the rotation axis with the $x$-axis,
 allowing us to derive one solution to \eqref{secod-ode} as
 $$ \psi=\frac{-k_g}{\sqrt{\mu}}.$$
Assuming the starting point is at  $\theta=0$, at the endpoint of the initial period, the progression angle $\theta(T,\lambda)$ monotonically increases from $\displaystyle 2\sqrt{2/3}\pi$ to $2\pi$
as $\lambda$ varies from $\displaystyle 3\left(\mu/2\right)^{2/3}$
to $+\infty$.
Here $T$ denotes the period and is dependent on $\lambda$.

Abresch and Langer \cite{al} studied the closed soliton solutions of normalized curve shortening flow by adding a tangential field to maintain constant speed along the curve in Euclidean plane. Analogously, Lopez De Lima and Montenegro \cite{lm} renormalized
the flow by adjusting the tangent component to  preserve constant affine speed.  Notably, in \cite{al,lm}, the rational number $q/p$ was confined to the open interval $\displaystyle (1/2,~\sqrt{2}/{2})$.
In the context of curve flow in centro-affine geometry, explored further in \cite{jiang2023, oqy, yang2023maximal}, the rational ratio  $q/p$ expands its scope to encompass the union of two open intervals $\displaystyle\left(0,\;1/2\right)\cup\left(1/2,\;+\infty\right)$ \cite{niu2025}. Transitioning to the focus of this paper, which examines solitons on the light-cone, we find that the rational ratio $q/p$ resides in the specific open interval  $\left(\sqrt{2/3}, 1\right)$ (refer to Section \ref{subsec-time} for more details).
\begin{thm}
Let $\mathbf{x}$ be a closed solitons for the flow \eqref{soflow} on the light-cone. Then we have the following possibilities for $\mathbf{x}$:
 \begin{itemize}
   \item[(1)] $\mathbf{x}$ is a planar ellipse;
   \item[(2)] $\mathbf{x} = \mathbf{x}_{p,q}$ has rotation index $p$ and closes up in $q$ periods of its curvature function.
    The pair $(p, q)$ is not arbitrary and must be such that $p/q$ is defined in the open interval $\left(\sqrt{2/3}, 1\right)$.
 \end{itemize}
\end{thm}
\begin{figure}[hbtp]
	\centering
		\includegraphics[width=.30\textwidth]{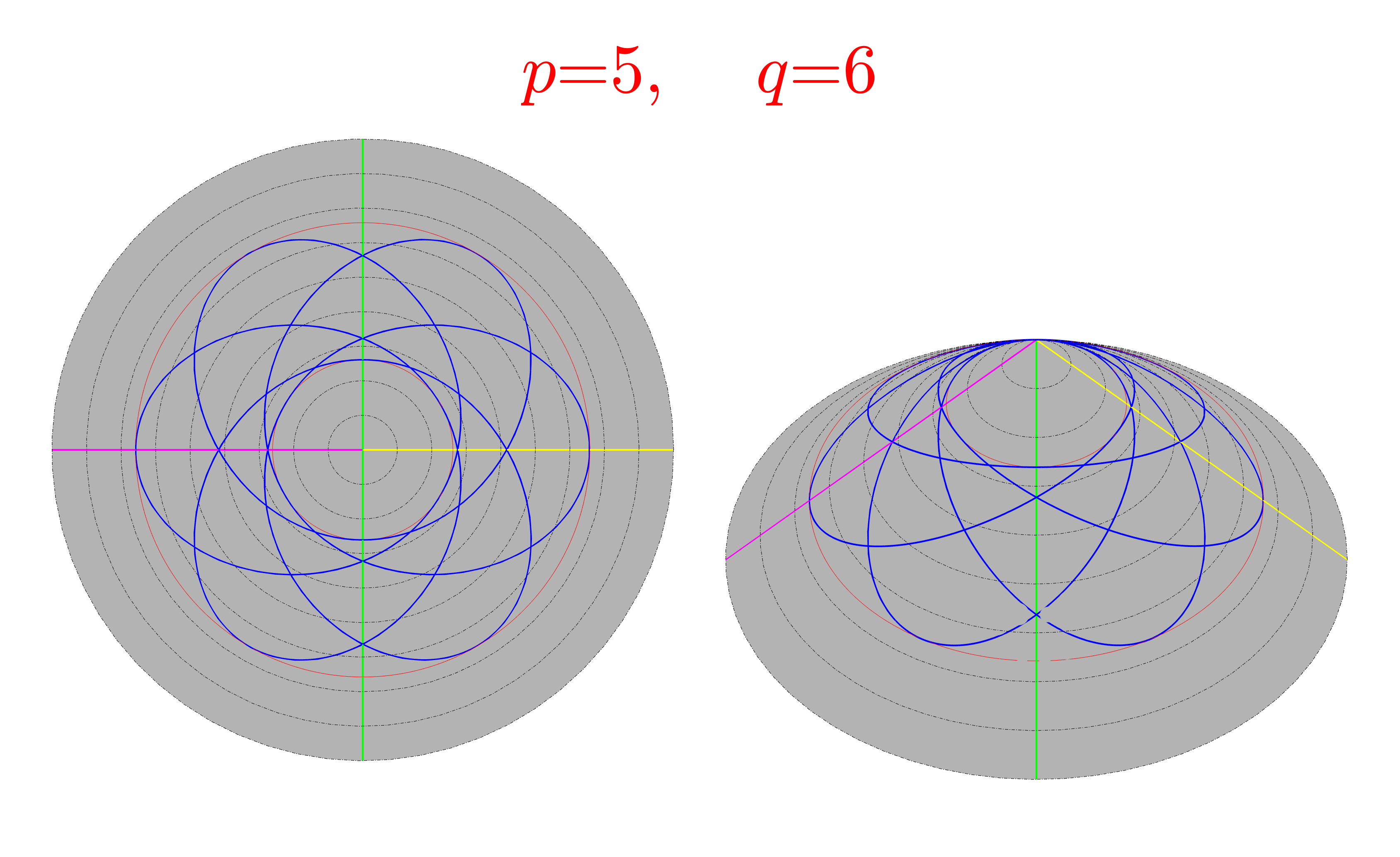}
        \includegraphics[width=.30\textwidth]{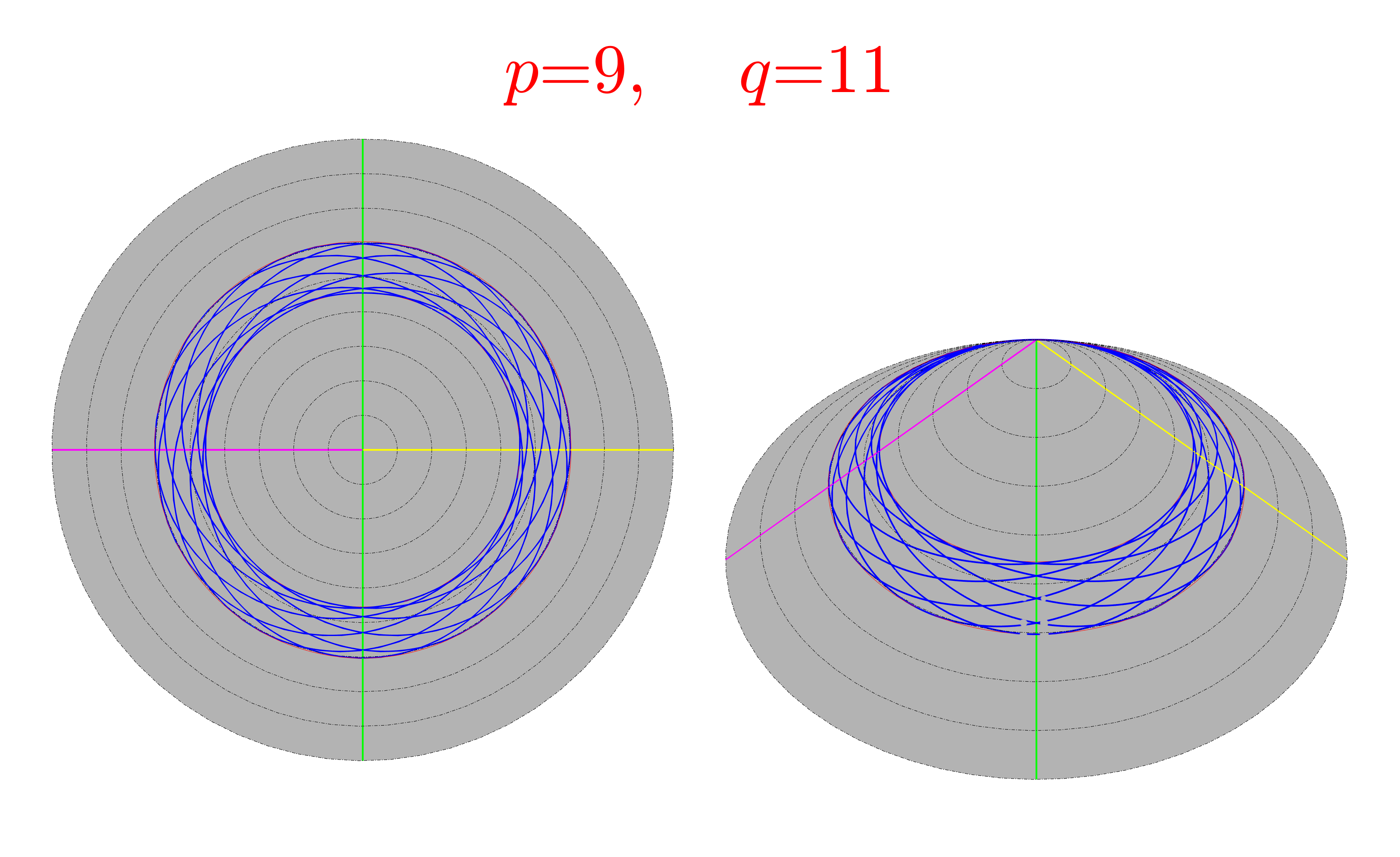}
        \includegraphics[width=.30\textwidth]{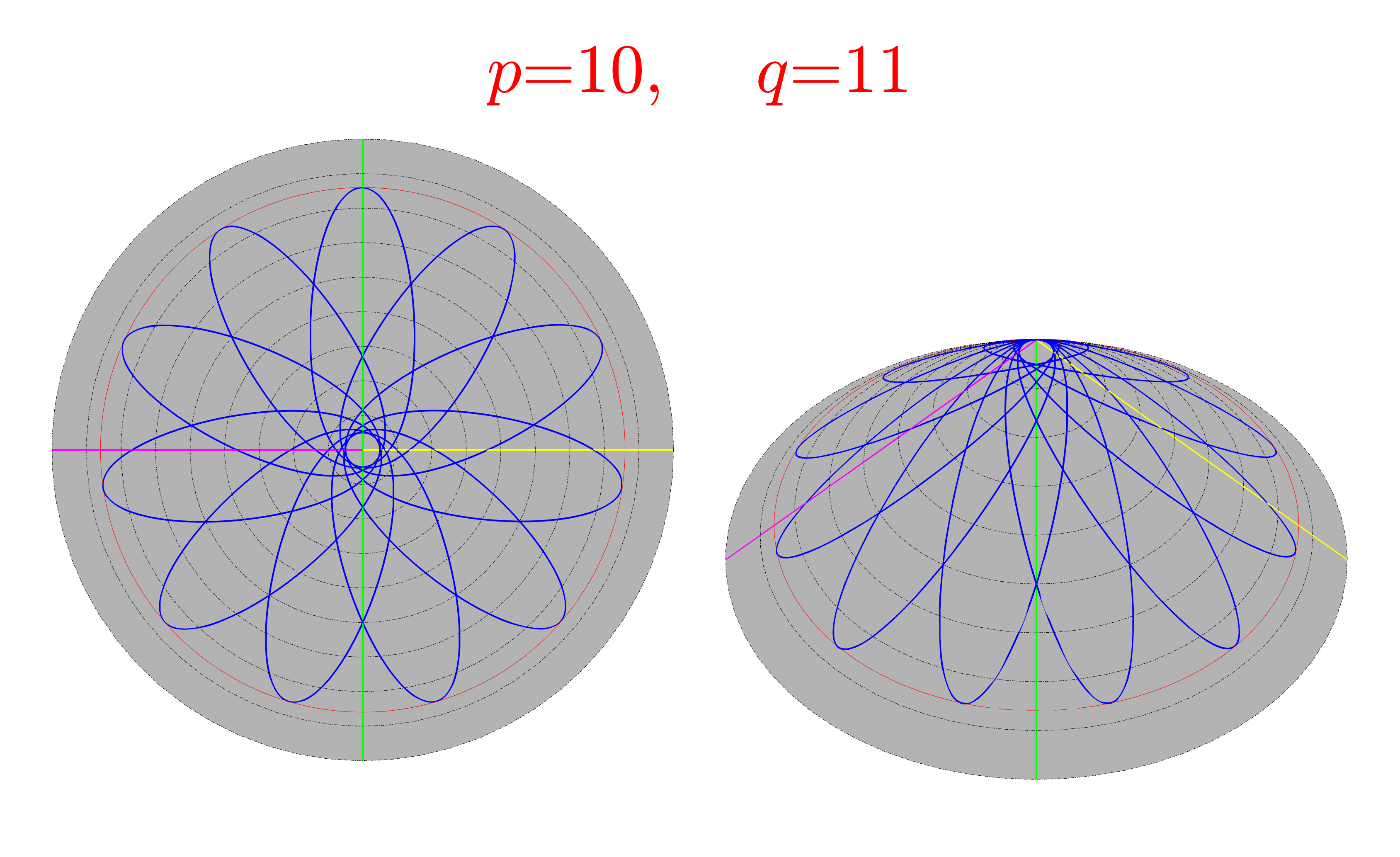}
        \\ \ \ \\
        \includegraphics[width=.30\textwidth]{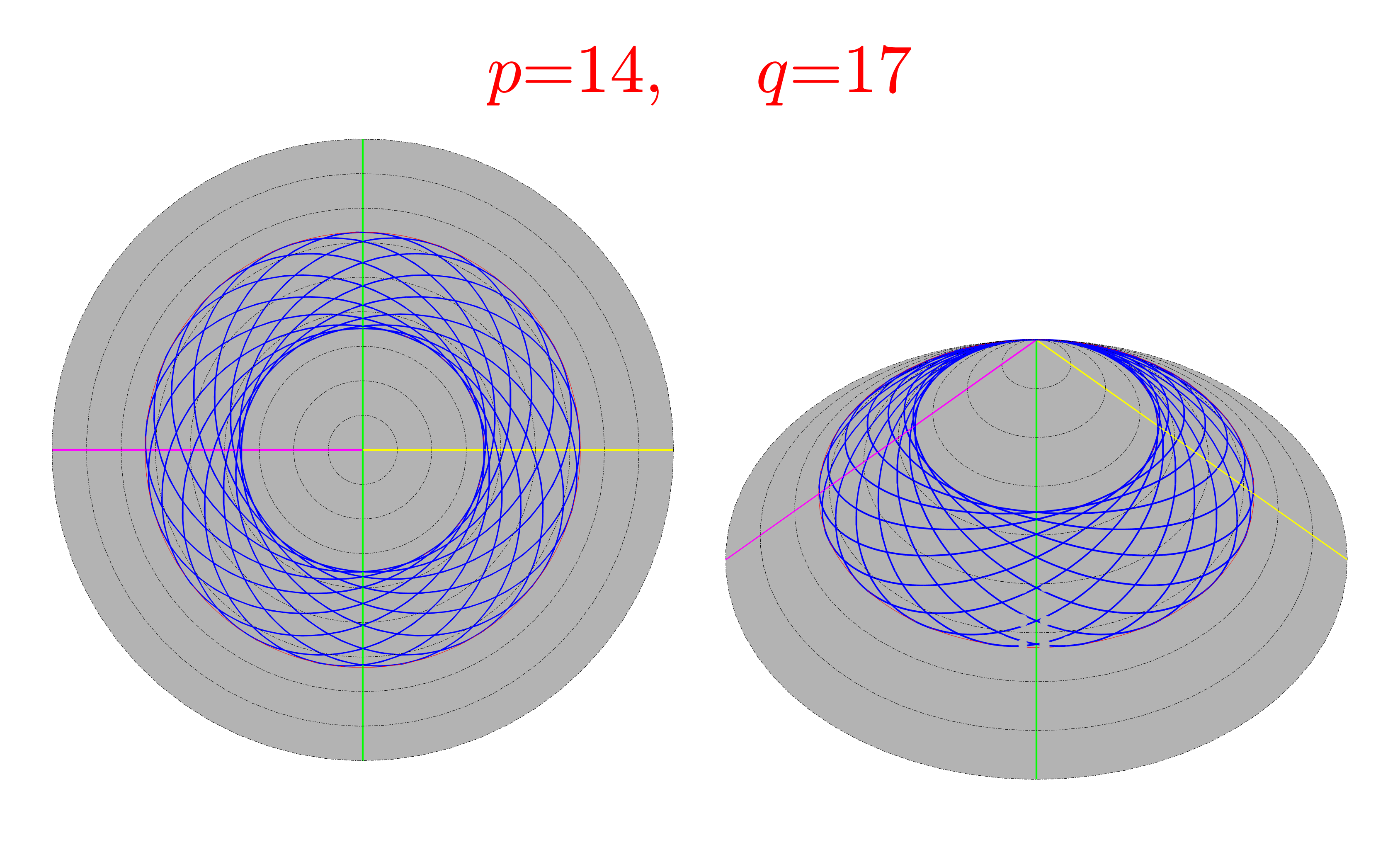}
        \includegraphics[width=.30\textwidth]{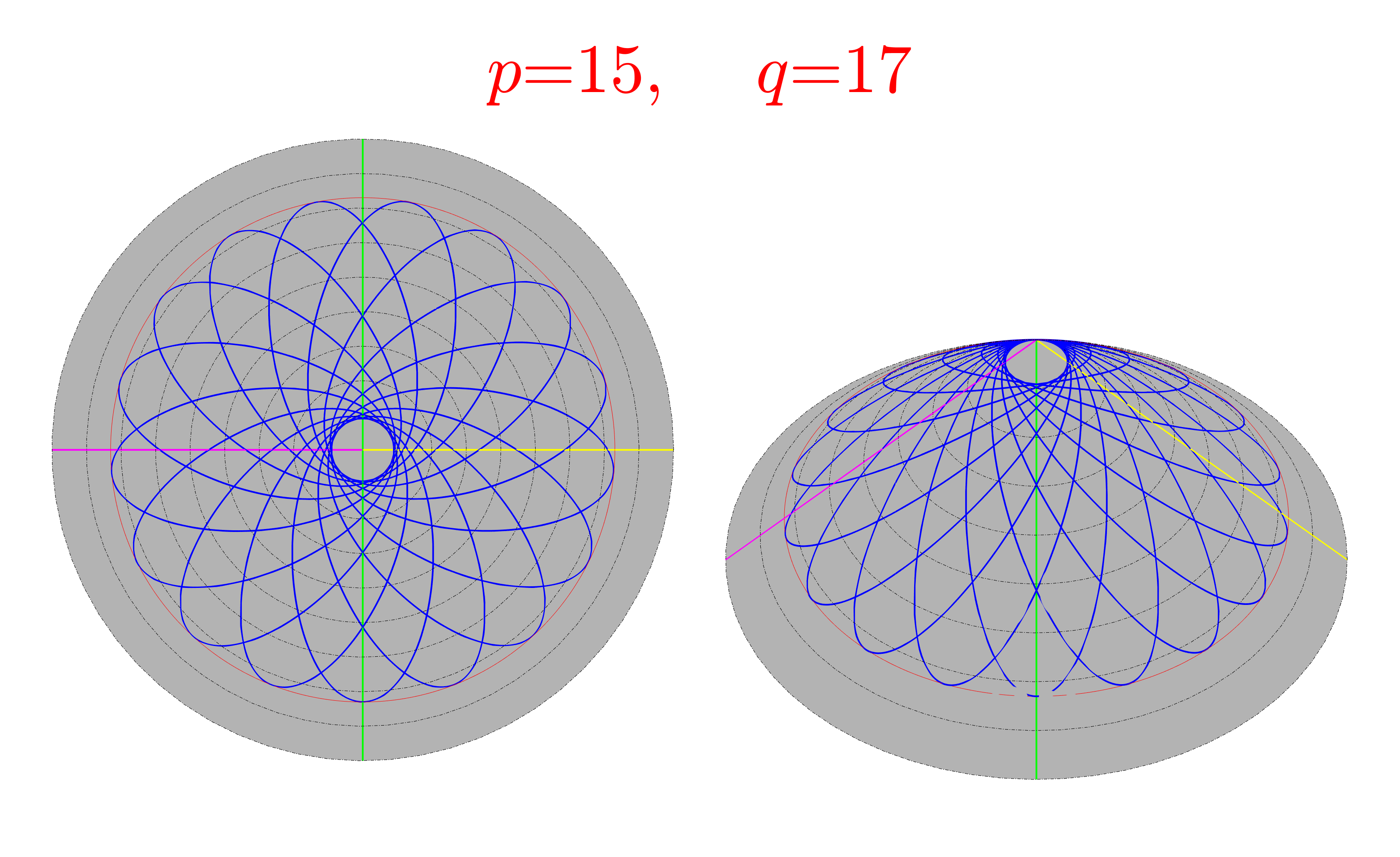}
        \includegraphics[width=.30\textwidth]{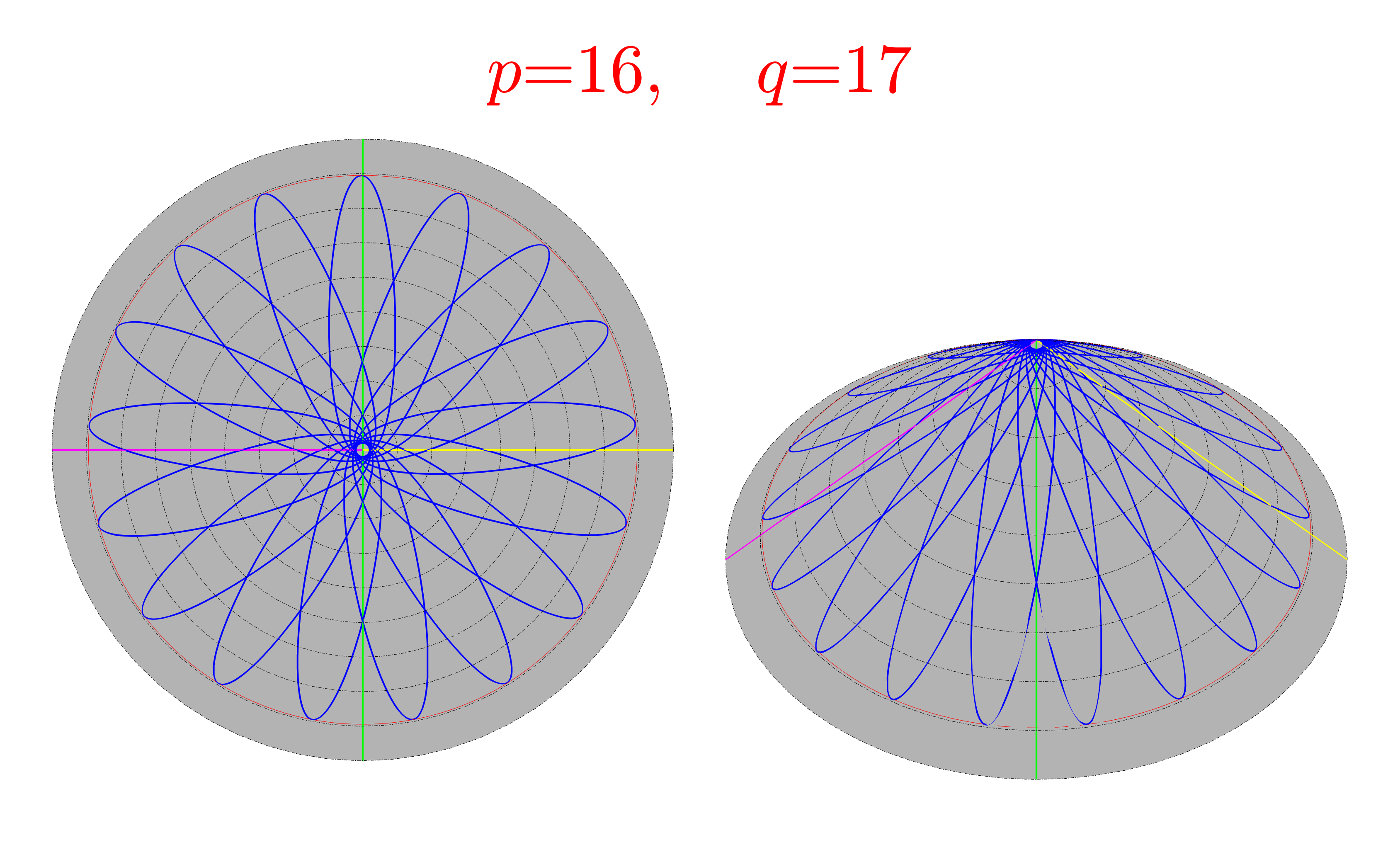}
	\setlength{\belowdisplayskip}{3pt}
	\caption{The closed solitons on light-cone.}
    \label{Fig:closed}
\end{figure}

In Fig. \ref{Fig:closed}, several examples of closed solitons are presented. These closed curves exhibit a circular motion up and down within the range  $\displaystyle \frac{-x_2}{\sqrt{\mu}}\leq\psi\leq\frac{-x_1}{\sqrt{\mu}}$.
The illustration demonstrates that as $p/q$ approaches $\sqrt{2/3}$, the difference  $x_2-x_1$ tends to $0$. Conversely, as $p/q$ approaches $1$, the difference $x_2-x_1$  tends to $+\infty$.

For $\mu=0$,  an appropriate rotation can align the rotation axis with the vector  $(1,1,0)$.
By setting the initial conditions $\psi(0)=-x_1$ and $\psi_s(0)=0$, we obtain another solution to \eqref{secod-ode} given by
$$ (1-\cos\theta)\psi=-2k_g$$ with the relation
$$\frac{d\theta}{ds}=\frac{1}{\psi}.$$
Due to the presence of divergent improper integrals at singular points, we are unable to derive a specific analytical expression for the variation of the angle $\theta$.
Nevertheless, numerical calculations (depicted in the left-hand side of  Fig. \ref{Fig:0lambda}) reveal that when the starting point is set at $\psi=-x_1$ and $\theta=\pi$, the angle $\theta(T,\lambda)$
 at the endpoint of the initial first period strictly monotonically increases from $2\pi$ to $3\pi$ as $\lambda$ varies from $0$ to $+\infty$.
This enables us to clearly comprehend the potential behavior of this type of soliton solution on the light-cone.

For $\mu<0$,  a suitable rotation can align the rotation axis with the $z$-axis, yielding another solution to  \eqref{secod-ode} is $$\psi\sin\theta=\frac{-k_g}{\sqrt{-\mu}}$$ and $$ \frac{d\theta}{ds}=\frac{1}{\psi}.$$
Again, due to the divergent improper integrals at singular points, we are unable to derive an explicit analytical expression for the variation of $\theta$.
Similarly, numerical calculations (depicted in the right-hand side of  Fig. \ref{Fig:0lambda}) offer further insights: when the starting point is set at $\psi=-x_1$ and $\displaystyle \theta=\frac{\pi}{2}$, the angle $\theta(T,\lambda)$
 at the endpoint of the initial first period strictly monotonically increases from $2\pi$ to $\frac{5\pi}{2}$ as $\lambda$ varies from $3\left(\mu/2\right)^{2/3}$ to $+\infty$.
This also allows us to clearly understand the potential behavior of this type of soliton solution on the light-cone.

\begin{figure}[hbtp]
	\centering
        \includegraphics[width=.30\textwidth]{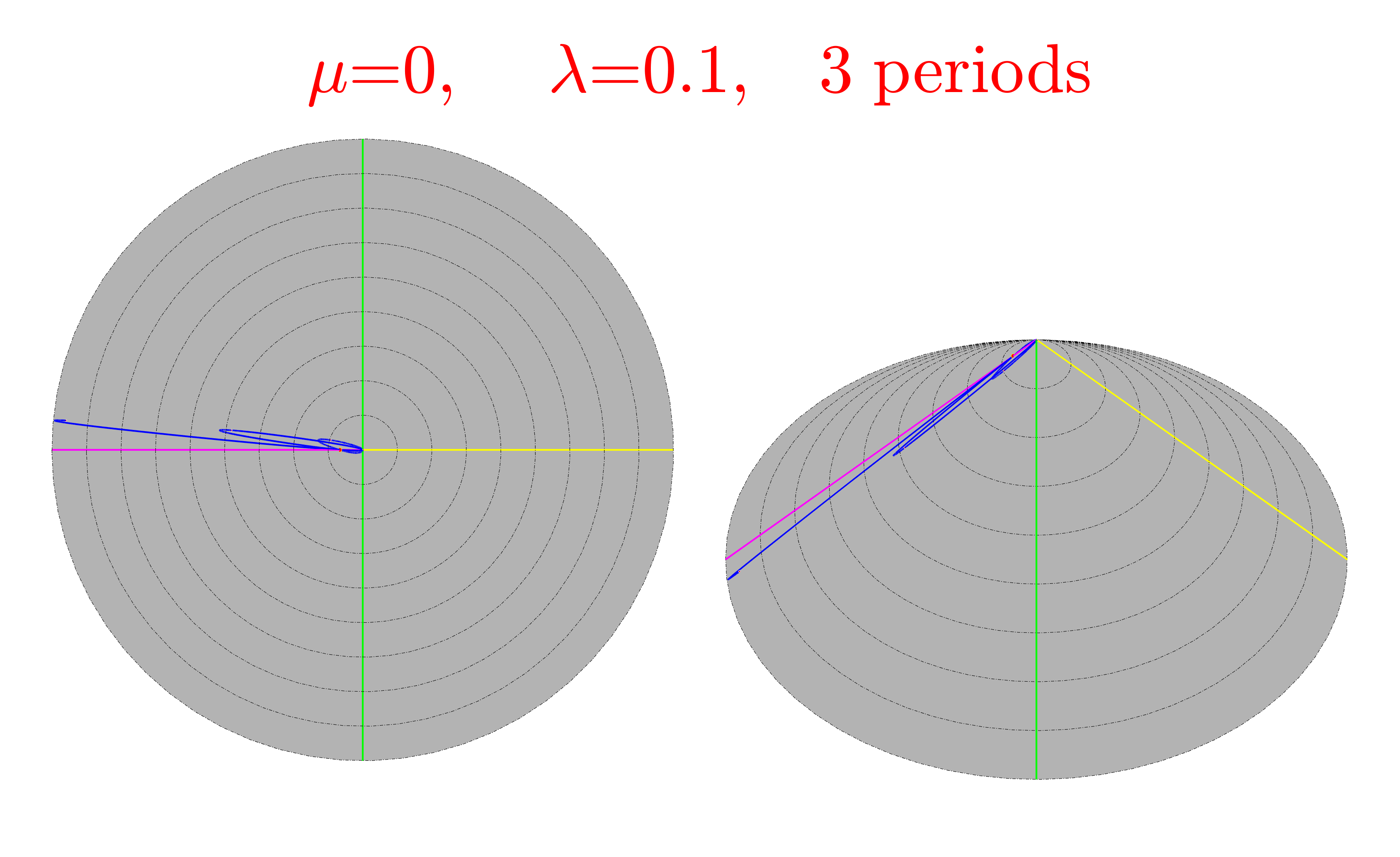}
        \includegraphics[width=.30\textwidth]{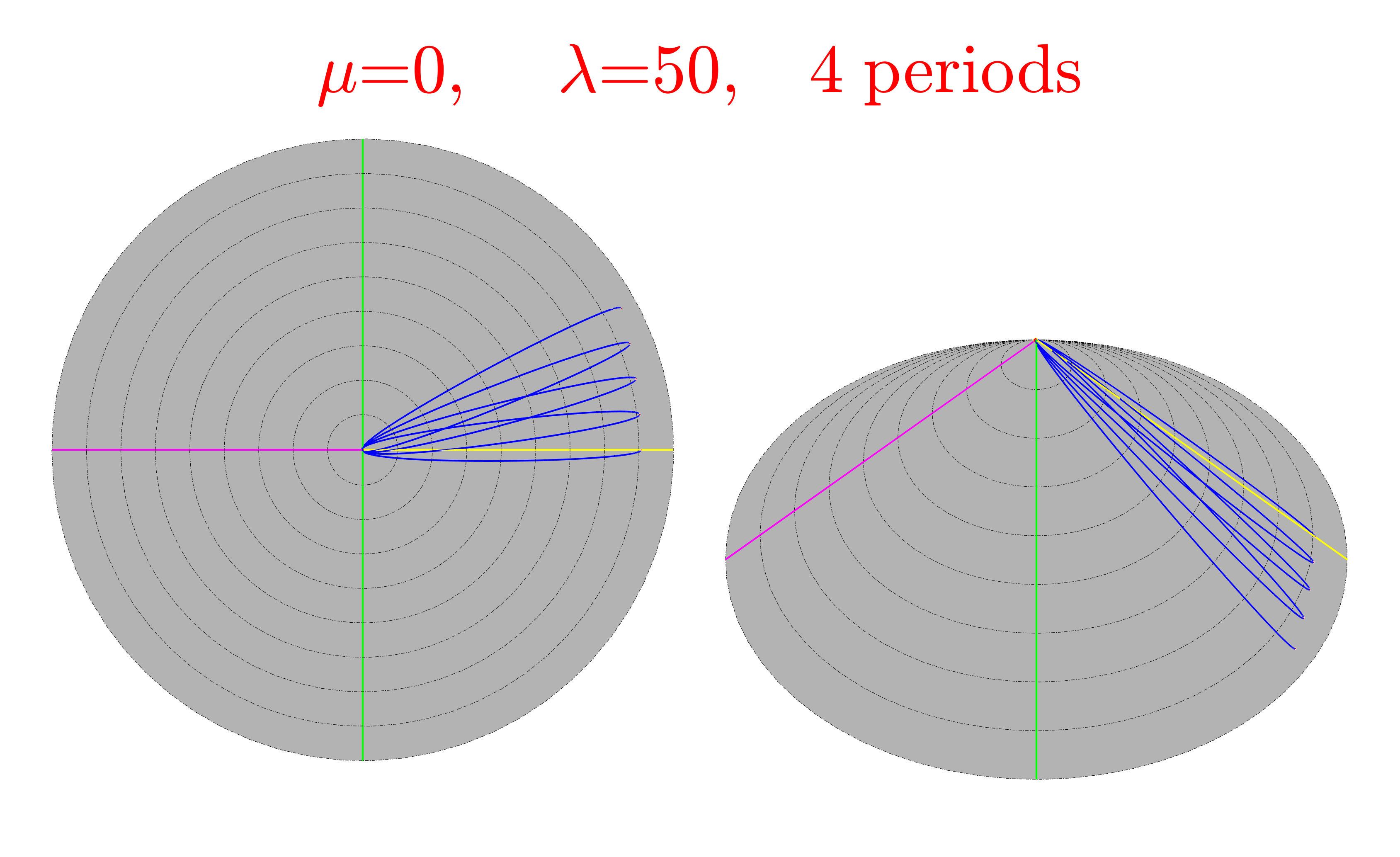}
        \includegraphics[width=.30\textwidth]{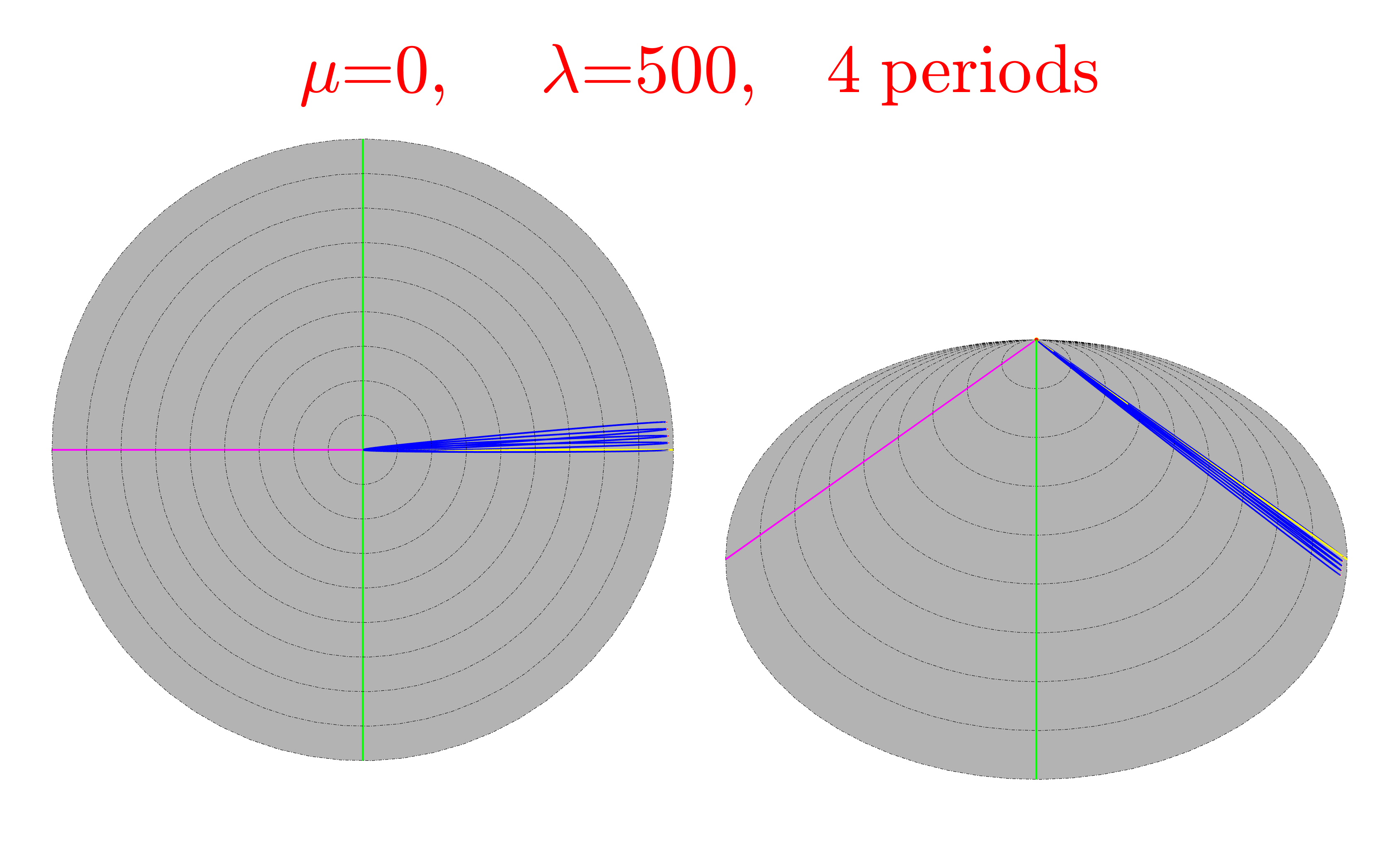}\\\ \ \\
        \includegraphics[width=.289\textwidth]{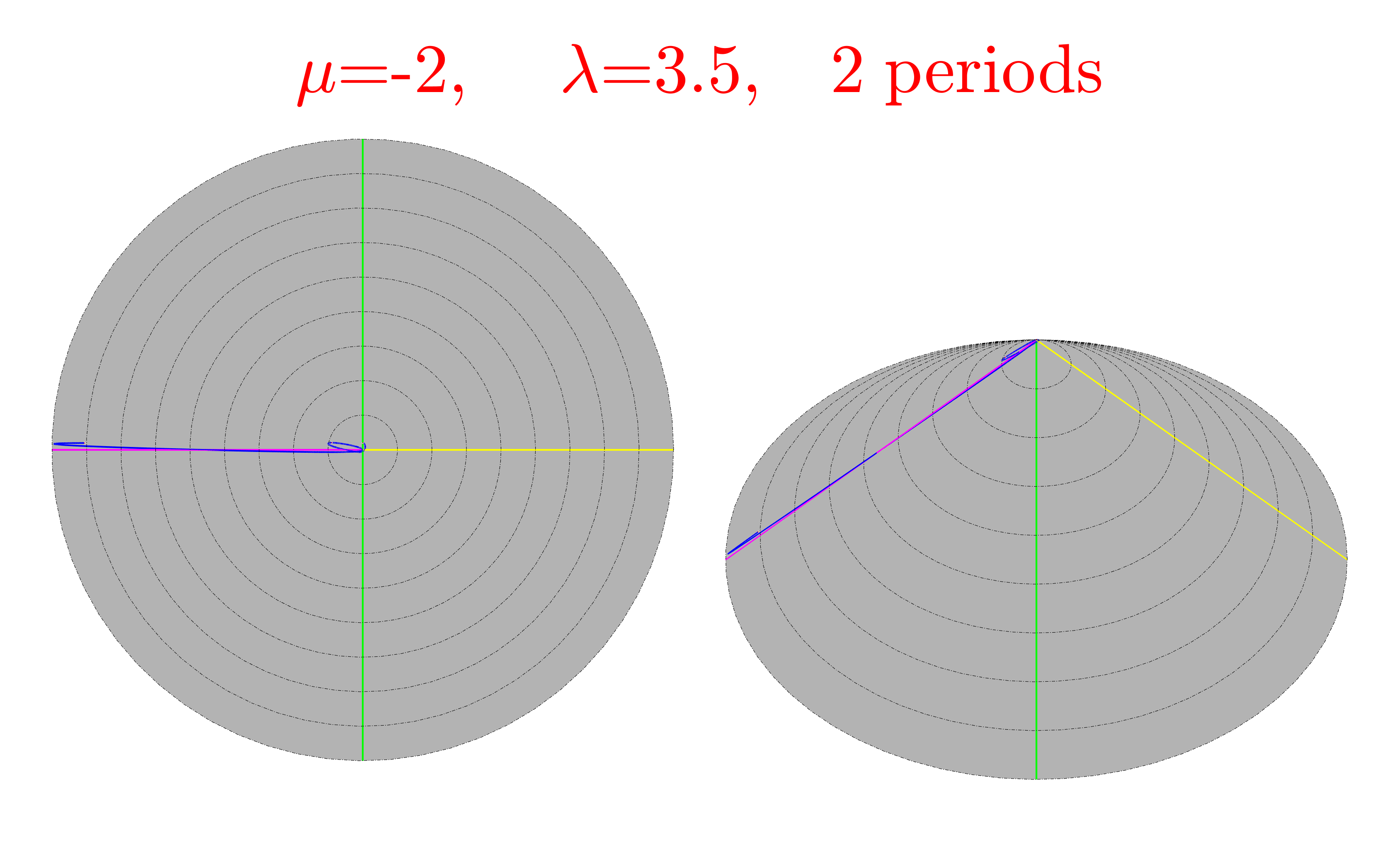}\;		
        \includegraphics[width=.30\textwidth]{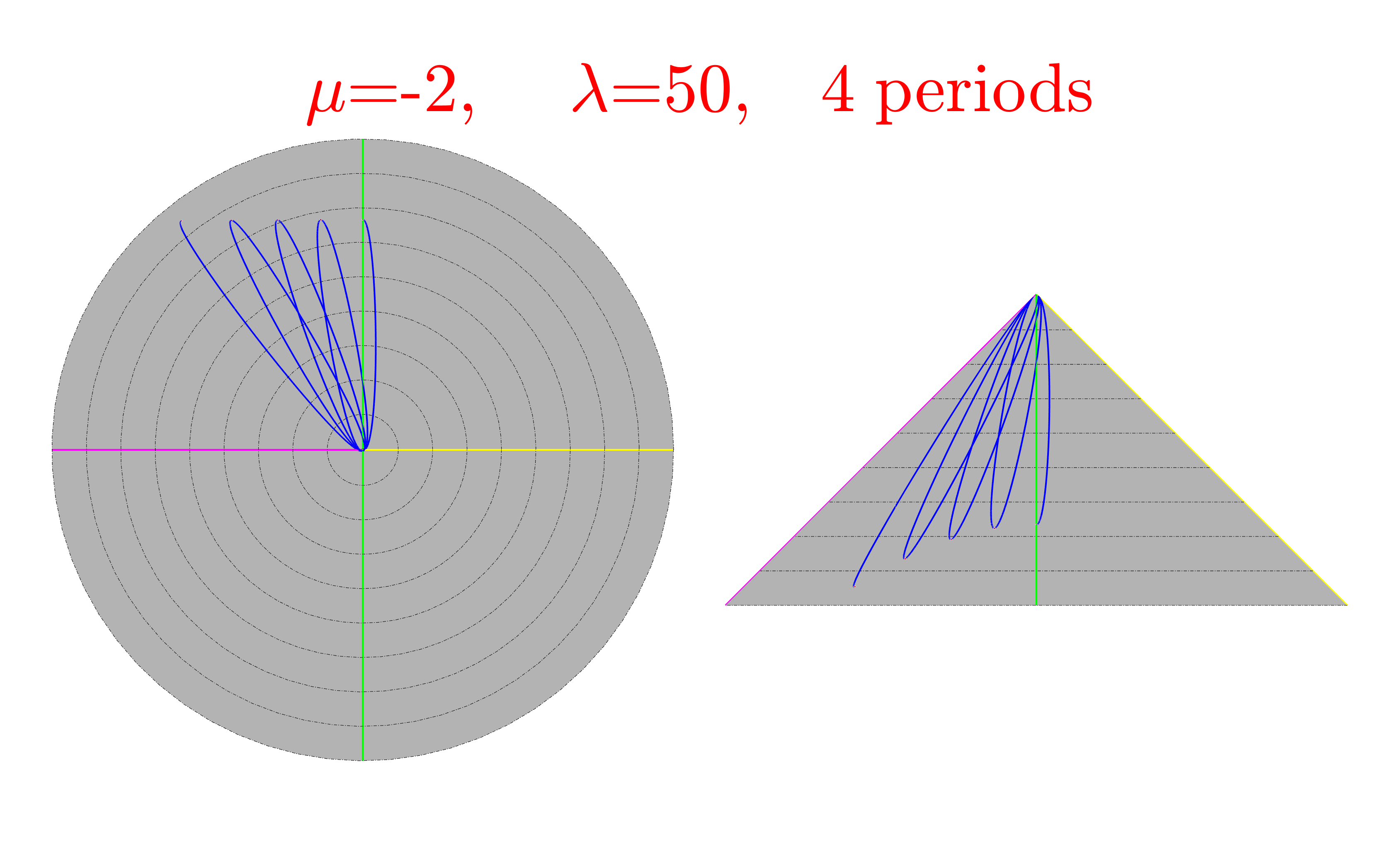}
        \includegraphics[width=.30\textwidth]{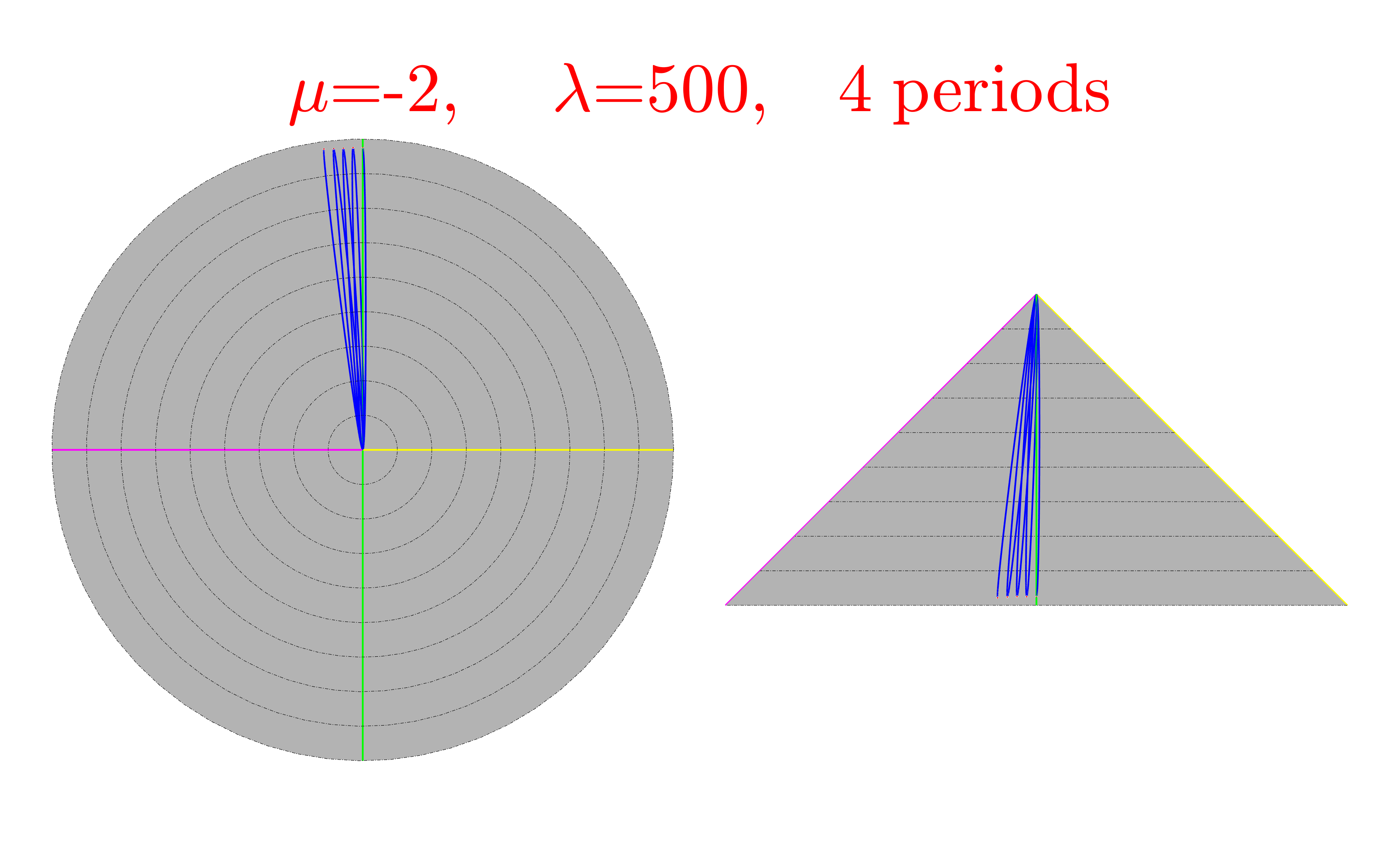}
	\setlength{\belowdisplayskip}{3pt}
	\caption{The non-closed periodic solitons on light-cone.}
    \label{Fig:nonclosed}
\end{figure}

Fig. \ref{Fig:nonclosed}  provides a visual representation of  solitons rotating around the light-like axis $(1,1,0)$ and the $z$-axis for $\mu=0$ and $\mu=-2$, respectively.
As detailed in Corollary \ref{cor-light} and Corollary \ref{cor-space}, the endpoints of each periodic cycle are situated at $\displaystyle (2n\pi, 3n\pi)$ and $\displaystyle \left(2n\pi, \frac{(4n+1)\pi}{2}\right)$,  respectively.
As $n$ increases, the endpoints approach $2n\pi$.
In the first row of Fig. \ref{Fig:nonclosed}, the magenta line represents the rotation axis $(1,1,0)$ with a fixed point situated at  $\displaystyle\psi=\frac{1}{\lambda}$.
Additionally, the yellow line is positioned at $\theta=\pi$.
The pattern is very evident: as  $\lambda$ diminishes towards $0$, the endpoints of the trajectories draw nearer to the magenta line. Conversely, as $\lambda$ augments, The endpoint of each cycle moves closer to the yellow line.
A comparable trend is clear in the second row of  Fig. \ref{Fig:nonclosed}:  as  $\lambda$ diminishes towards $3$, the endpoints of the trajectories draw nearer to the magenta line. Conversely, as $\lambda$ increases, the endpoint of each cycle migrates towards the starting point of the trajectories.

With the hints from the previous three special solutions, for \eqref{secod-ode}, we can present the general form of the solution (see Section \ref{subsec-ode} for the proof).
\begin{thm}\label{thm-ode}
Consider three constants  $x_1$, $x_2$ and $x_3$ that adhere to the conditions $x_1+x_2+x_3=0$ and $x_1<x_2<x_3$. Define the function
$$f(s)=x_1+(x_2-x_1)\mathrm{sn}\left(\frac{\sqrt{x_3-x_1}}{2}s,\;\sqrt{\frac{x_2-x_1}{x_3-x_1}}\right)^2.$$
Then the second order differential equation
\begin{align*}
  \psi_{ss}-\frac{\psi_{s}^2+1}{2\psi}-f(s)\psi=0,
\end{align*}
admits the solution
\begin{align*}
  (-c_1+c_2\cos\theta+c_3\sin\theta)\psi=f(s),
\end{align*}
where the integral constants $c_1,c_2,c_3$ fulfill the relation $-c_1^2+c_2^2+c_3^2=-x_1x_2x_3$ , and $\displaystyle\theta=\int_0^s\frac{dt}{\psi(t)}$.
\end{thm}

Specifically,  utilizing the curvature formula provided in \eqref{rem-k_g},
\eqref{secod-ode} can be reconfigured into
\begin{align*}
 2\psi_{\theta\theta}\psi-3\psi^2_{\theta}-2\psi^5(-c_1+c_2\cos\theta+c_3\sin\theta)-\psi^2=0.
\end{align*}
Based on the comprehensive derivation process outlined previously, it becomes evident that the solutions to this equation over the entire interval
can be constructed by rotating the points $(\psi(\theta),\psi(\theta)\cos\theta, \psi(\theta)\sin\theta)$ in the first period around the axis $(c_1,c_2,c_3)$.
Consequently, by substituting $\psi=u^{-2}$, and combining Proposition \ref{prop-b-lignt} and Proposition \ref{prop-b-space},   the following corollary emerges.
\begin{cor} For the ordinary differential equation
\begin{align*}
 u_{\theta\theta}+\frac{1}{2u^{5}}(-c_1+c_2\cos\theta+c_3\sin\theta)+\frac{u}{4}=0,
\end{align*}
where the constants $c_1,c_2$ and $c_3$ cannot all be zero simultaneously, the positive solution $u(\theta)$ demonstrates the subsequent properties.
\begin{itemize}
  \item If $-c_1^2+c_2^2+c_3^2<0$, there exist two positive constant $D_1$ and $D_2$ such that $D_1\leq u(\theta)\leq D_2$.
  \item If $-c_1^2+c_2^2+c_3^2\geq0$, $u(\theta)$ is unbounded and its infimum is $0$.
\end{itemize}
 Notably, when $-c_1^2+c_2^2+c_3^2=0$,
 there exists a fixed point $u(\theta_0)=u(\theta_0+2\pi)$, where $\theta_0$ satisfies the conditions $c_1\sin\theta_0=c_2$ and $c_1\cos\theta_0=c_3$.
\end{cor}
\subsection*{Outline}  In Section \ref{sec-back}, we dedicate some time to revisiting the geometry of the light-cone and introducing our notations. Moving on to Section \ref{sec-ev}, we derive the relevant evolution equations, ultimately establishing the Li-Yau Harnack inequality for heat flow and highlighting some fascinating properties related to a third-order flow. Then, in Section \ref{sec-soliton}, we deduce the space-periodic solitons for this third-order flow and subsequently obtain analytic solutions to a second-order nonlinear ordinary differential equation.

\subsection*{Acknowledgements} This work was supported by Foreign Expert Key Support Program (Northeast Special) Grants-D20240219.


\section{Preliminaries}\label{sec-back}
Let $V$ be an $n$-dimensional real vector space. A {\it Lorentzian scalar product} on $V$ is defined as a non-degenerate symmetric bilinear form
$\langle\cdot,\cdot\rangle_L$ of index $1$. This implies the existence of a basis $\mathbf{e}_1,\cdots,\mathbf{e}_n$ of $V$
such that
\begin{align*}
  \langle\mathbf{e}_i,\mathbf{e}_j\rangle_L=\left
                \{\begin{array}{cl}
                     -1, & \mathrm{if}\quad i=j=1,  \\
                     1, & \mathrm{if}\quad i=j=2,\ \cdots,\ n, \\
                     0, & \mathrm{otherwise}.
                   \end{array}
  \right.
\end{align*}
A {\it Lorentzian manifold} is a pair $(M, g)$, where $M$ is an $n$-dimensional smooth
manifold and $g$ is a Lorentzian metric, i.e., $g$ associates to each point $p\in M$ a
Lorentzian scalar product $g_p$ on the tangent space $T_pM$.

The {\it Minkowski space} $\E_1^n$ is the vector space $\R^n$ endowed with the Lorentzian scalar product defined as follows:
\begin{equation}\label{inn-p}
\langle\mathbf{u},\mathbf{v}\rangle_L=-u_1v_1+u_2v_2+\cdots+u_nv_n,
\end{equation}
where $\mathbf{u}=(u_1,\cdots,u_n)\in\R^n$ and $\mathbf{v}=(v_1,\cdots,v_n)\in\R^n$.
In cartesian coordinates $(x_1,\cdots , x_n)$ on $\R^n$, the {\it Minkowski metric}
is defined by $\displaystyle g=-(dx_1)^2+(dx_2)^2+\cdots+(dx_n)^2$.

In Minkowski space $\E_1^n$, a vector $\mathbf{v}=\{v_1,\cdots,v_n\}\in \E_1^n$ is classified as follows:
\begin{itemize}
  \item It is called {\it space-like} if$\displaystyle\langle \mathbf v,\mathbf v\rangle_L >0$ or if $\displaystyle\mathbf v=\mathbf 0$.
  \item It is termed {\it time-like} if $\displaystyle\langle \mathbf v,\mathbf v\rangle_L< 0$.
  \item It is designated as {\it light-like} (or {\it null}) if $\displaystyle\langle \mathbf v,\mathbf v\rangle_L= 0$ and $\displaystyle\mathbf v\neq\mathbf 0$.
  \item It is considered {\it causal}, if it is time-like or space-like.
\end{itemize}
The magnitude (or length) of the vector $\mathbf v$  is denoted by $\displaystyle|\mathbf v|$ and is defined as by $\displaystyle|\mathbf v|=\sqrt{|\langle \mathbf v,\mathbf v\rangle_L|}$.

Furthermore, for a curve $\displaystyle \mathbf r = \mathbf r(p):I\subset\R\to \E^n_1$, it is
classified as space-like, time-like, or light-like, if its velocity vector $\mathbf r'(p)$ is space-like, time-like, or light-like, respectively.

The {\it plane} with {\it pseudo-normal} $\mathbf{v}$ in $\E^n_1$ is defined by
\begin{equation*}
  P(\mathbf{v},c)=\{\mathbf{x}\in\E^n_1\;|\;\langle\mathbf{x},\mathbf{v}\rangle_L=c\},
\end{equation*}
where $\mathbf{v}$ is a non-zero vector in $\E^n_1$ and $c$ is constant. Moreover, $P(\mathbf{v},c)$ is called a space-like plane, a time-like plane or light-like plane if
 $\mathbf{v}$ is time-like, space-like or light-like, respectively.

Given a real number $c>0$, we define the following subspaces in Minkowski space $\E^n_1$.
\begin{itemize}
  \item The {\it hyperbolic space} $H^{n-1}(c)$ is given by
  \begin{equation*}
     H^{n-1}(c)=\{\mathbf{x}\in\E^{n-1}_1 \;|\; \langle\mathbf{x},\mathbf{x}\rangle_L=-c^2\}.
  \end{equation*}
  \item The {\it de Sitter space} $S^{n-1}_1(c)$ is defined as
   \begin{equation*}
      S^{n-1}_1(c)=\{\mathbf{x}\in\E^n_1 \;|\; \langle\mathbf{x},\mathbf{x}\rangle_L=c^2\}.
    \end{equation*}
  \item The {\it light-cone} $LC$ is described by
  \begin{equation*}
    LC=\{\mathbf{x}\in\E^n_1\backslash\{0\} \;|\; \langle\mathbf{x},\mathbf{x}\rangle_L=0\}.
  \end{equation*}
\end{itemize}
In $\E^n_1$, the term pseudo-sphere collectively refers to the hyperbolic space, the de Sitter space, and the light-cone.

The {\it arc length parameter} of a curve
$\mathbf{r}(p)\in\E^n_1$ is defined by
\begin{equation*}
  \frac{ds}{dp}=\left|\frac{d\mathbf{r}}{dp}\right|.
\end{equation*}
 Assume $\mathbf{r}$ be a curve in $\E^n_1$ parametered with arc length $s$, the unit tangent vector $\mathbf{T}$ is denoted by $\mathbf{T}=\frac{d\mathbf{r}}{ds}=\mathbf{r}_s$.
In $\E_1^3$, the pseudo vector product of $\mathbf{u}$ and $\mathbf{v}$ is given by
\begin{equation}\label{cross-p}
\mathbf{u}\times_L\mathbf{v}=(u_3v_2-u_2v_3,\, u_3v_1-u_1v_3,\, u_1v_2-u_2v_1).
\end{equation}

\subsection{The Frenet formulas in Minkowski space}
In $\E_1^3$, based on the characteristics of $\mathbf{T}_s$,  we categorize our discussion into the following two scenarios.

 {\bf Case 1.} If $\langle\mathbf{T}_s, \mathbf{T}_s\rangle_L\neq0$,
  the unit principal normal vector field is defined as $\mathbf{N}=\frac{\mathbf{T}_s}{\vert  \mathbf{T}_s\vert}$, and the unit binormal vector is given by $\mathbf{B}= \mathbf{T}\times_L  \mathbf{N}$. Here, $\langle  \mathbf{T}, \; \mathbf{T}\rangle_L=\eta_1$, $\langle \mathbf{N},\; \mathbf{N}\rangle_L=\eta_2,$ and $\langle \mathbf{B},\;\mathbf{B}\rangle_L=\eta_3,$ where $\eta_i=\pm 1,\; i=1,2,3$. Note that according to \eq{inn-p} and \eq{cross-p}, we have
 \begin{equation*}
  \eta_1\eta_2\eta_3=-1,\quad \eta_1+\eta_2+\eta_3=1.
 \end{equation*}
 The Frenet formula is expressed as (see \cite{wj} for more details)
\begin{equation}\label {f-eq}
	\left[
	\begin{array}{c}
		 \mathbf{T}_s\\
		\mathbf{N}_s \\
		\mathbf{B}_s
	\end{array}
	\right]
	=
	\left[
	\begin{array}{ccc}
		0 & \kappa & 0 \\
		\eta_3\kappa & 0 &  \tau \\
		0 & \eta_1 \tau & 0
	\end{array}
	\right]
	\left[
	\begin{array}{c}
		 \mathbf{T} \\
		\mathbf{N} \\
		\mathbf{B}
	\end{array}
	\right],
\end{equation}
where $\kappa,\tau$ are referred to as the curvature and torsion of  $\mathbf r$, respectively.

{\bf Case 2.}  If $\langle\mathbf{T}_s, \mathbf{T}_s\rangle_L=0$, the principal normal vector field $\mathbf{N}$ is defined as $\mathbf{T}_s$. The binormal vector field $\mathbf{B}$ is the unique light-like vector field perpendicular to
$T$ such that $\langle N,B\rangle_L=1$. The Frenet formula  in this case is given by (refer to \cite{wj} for further information)
\begin{equation}\label {f-eq-light}
	\left[
	\begin{array}{c}
		 \mathbf{T}_s\\
		\mathbf{N}_s \\
		\mathbf{B}_s
	\end{array}
	\right]
	=
	\left[
	\begin{array}{ccc}
		0 & 1 & 0 \\
		0 & \bar{\tau} &  0 \\
		-1 & 0 & -\bar{\tau}
	\end{array}
	\right]
	\left[
	\begin{array}{c}
		 \mathbf{T} \\
		\mathbf{N} \\
		\mathbf{B}
	\end{array}
	\right].
\end{equation}
\subsection{The Frenet formulas on light-cone}
 Let us consider the Frenet-Serret type formulae of curve $\mathbf{r}$ on light-cone $LC$ in $\E_1^3$.
 In this paper, our primary focus is on curves that lie on the light-cone. It is important to highlight that vectors orthogonal to a light-like vector can only be either light-like or space-like. Furthermore, two light-like vectors are orthogonal to each other if and only if they are linearly dependent. Therefore, throughout the following sections, we will always assume that the unit tangent vector  $\mathbf{T}$ is space-like and $\displaystyle\frac{ds}{dp}\neq0$.

Since $\mathbf{r}(p)$ is the light-like vector field, $\displaystyle \frac{d\mathbf{r}}{dp}$ is a space-like vector field perpendicular to $\mathbf{r}(p)$. The normal vector field $\mathbf{Y}$ is the unique light-like vector field perpendicular to $\mathbf{T}$ such that $\langle\mathbf{r},\mathbf{Y}\rangle_L=1$. Then the Frenet formula can be expressed by (see \cite{liu2004} for more details)
\begin{equation}\label{FS-eq-light}
	\left[
	\begin{array}{c}
		 \mathbf{r}_s\\
		\mathbf{T}_s \\
		\mathbf{Y}_s
	\end{array}
	\right]
	=
	\left[
	\begin{array}{ccc}
		0 & 1 & 0 \\
	k_g & 0 &  -1 \\
		0 & -k_g & 0
	\end{array}
	\right]
	\left[
	\begin{array}{c}
	\mathbf{r} \\
		\mathbf{T} \\
		\mathbf{Y}
	\end{array}
	\right],
\end{equation}
where $k_g=\langle\mathbf{r}_{ss},\mathbf{Y}\rangle_L.$

If $\langle\mathbf{T}_s, \mathbf{T}_s\rangle_L\neq0$, by \eq{f-eq} and \eq{FS-eq-light}, we derive the following relationships
\begin{equation*}
  k_g=-\eta_2\frac{k^2}{2},\qquad \tau^2\kappa^2=\kappa_s^2.
\end{equation*}
On the other hand, when $\langle\mathbf{T}_s, \mathbf{T}_s\rangle_L=0$, combining of \eq{f-eq-light} and \eq{FS-eq-light} gives
 $k_g=0$ and $\bar{\tau}=0$.

In fact, the right half of a light-cone, denoted as $LC^*$, can be represented using spherical coordinates as $\displaystyle (\psi, \psi\cos\theta, \psi\sin\theta)$ with the constraint $\psi>0$,
and the curve $\mathbf{r}$ on light-cone $LC^*$ can be described by $\psi=\psi(\theta)$. A direct computation leads to
\begin{align*}
\mathbf{T}&\;=\;\frac{1}{\psi}\left(\psi_{\theta}, \psi_{\theta}\cos\theta-\psi\sin\theta, \psi_{\theta}\sin\theta+\psi\cos\theta\right),\\
\mathbf{Y}&\;=\;\frac{1}{2\psi}\Big(-\left(1+\frac{\psi_{\theta}^2}{\psi^2}\right), \left(1-\frac{\psi_{\theta}^2}{\psi^2}\right)\cos\theta+\frac{2\psi_{\theta}}{\psi}\sin\theta, \\ &\qquad\qquad\qquad\qquad\qquad\qquad-\frac{2\psi_{\theta}}{\psi}\cos\theta+\left(1-\frac{\psi_{\theta}^2}{\psi^2}\right)\sin\theta\Big),
\end{align*}
and then
\begin{align}\label{rem-k_g}
\frac{ds}{d\theta}=\psi,\qquad k_g=-\frac{\psi^2+3\psi_{\theta}^2-2\psi_{\theta\theta}\psi}{2\psi^4}.
\end{align}
\begin{prop}\label{prop-plane}
 The curve $\mathbf{r}$ represents the intersection of a plane $P(\mathbf{v},1)$ and $LC^*$ if and only if the curvature $k_g$ remains constant, with its value given by $\displaystyle k_g=\frac{\langle \mathbf v,\mathbf v\rangle_L}{2}$.
 Specifically, if $\mathbf{r}$ is an ellipse, $k_g<0$; if $\mathbf{r}$ is a parabola, $k_g=0$; and if $\mathbf{r}$ is a hyperbola, $k_g>0$.
\end{prop}
\begin{proof}
  Assume $\mathbf{v}=(v_1,v_2,v_3)$, and then the equation of the plane is  $-v_1\psi+v_2\psi\cos\theta+v_3\sin\theta=1$. Solving for $\psi$,  we find
  \begin{equation*}
    \psi=\frac{1}{-v_1+v_2\cos\theta+v_3\sin\theta}.
  \end{equation*}
  Substituting this expression into the formula for the curvature \eqref{rem-k_g} and performing the necessary calculations, we obtain $\displaystyle k_g=\frac{\langle \mathbf v,\mathbf v\rangle_L}{2}$.

  On the other hand, the nature of $\mathbf{v}$ determines the type of curve $\mathbf{r}$. If $\mathbf{v}$ is time-like, $\mathbf{r}$ is an ellipse; if $\mathbf{v}$ is light-like, $\mathbf{r}$ is a parabola; and if $\mathbf{v}$ is space-like, $\mathbf{r}$ is a hyperbola.

Conversely, if $k_g$ is constant, we can select three constants $c_1$, $c_2$ and $c_3$ (not all zero) such that $\displaystyle -c_1^2+c_2^2+c_3^2=2k_g$. Then $\displaystyle\psi=\frac{1}{-c_1+c_2\cos\theta+c_3\sin\theta}$ is the general solution of  the curvature equation \eqref{rem-k_g}. This, in turn, implies that $\mathbf{r}$ is a planar curves.
\end{proof}

In $\E_1^3$, the rotation matrices around the $x$-axis, $z$-axis and light-like axis $(1,1,0)$ are respectively given by
\begin{eqnarray}\label{RM}
  \begin{aligned}
&\left(
  \begin{array}{ccc}
    1 & 0 & 0 \\
    0 & \cos\omega & -\sin\omega \\
    0 & \sin\omega & \cos\omega \\
  \end{array}
\right),\\ \\
&\left(
  \begin{array}{ccc}
    \cosh\omega & \sinh\omega & 0 \\
    \sinh\omega & \cosh\omega & 0\\
    0 & 0 & 1 \\
  \end{array}
\right),\\ \\
&\left(
  \begin{array}{ccc}
    1+\frac{\omega^2}{2} & -\frac{\omega^2}{2} & \omega \\
    \frac{\omega^2}{2} & 1-\frac{\omega^2}{2} & \omega\\
    \omega & -\omega & 1 \\
  \end{array}
\right).
\end{aligned}
\end{eqnarray}
The corresponding Killing vector fields are
\begin{equation*}
   -z\partial_y+y\partial_z,\qquad y\partial_x+x\partial_y, \qquad z\partial_x+z\partial_y+(x-y)\partial_z.
\end{equation*}
\begin{lem}\label{lem=kf}
  On light-cone $(\psi, \psi\cos\theta, \psi\sin\theta)$, the  vector fields
  $\partial_{\theta}$, $\psi\cos\theta\partial_{\psi}-\sin\theta\partial_{\theta}$, $\psi\sin\theta\partial_{\psi}-(1-\cos\theta)\partial_{\theta}$ are Killing vector fields.
\end{lem}
\begin{proof} A direct computation shows
\begin{equation*}
  \partial_{\psi}=\partial_x+\cos\theta\partial_y+\sin\theta\partial_z,\qquad \partial_{\theta}=-\psi\sin\theta\partial_y+\psi\cos\theta\partial_z.
\end{equation*}
On the other hand,
\begin{align*}
-z\partial_y+y\partial_z&\;=\;-\psi\sin\theta\partial_y+\psi\cos\theta\partial_z,\\
y\partial_x+x\partial_y&\;=\;\psi\cos\theta\partial_x+\psi\partial_y,\\
z\partial_x+z\partial_y+(x-y)\partial_z&\;=\;\psi\sin\theta(\partial_x+\partial_y)+\psi(1-\cos\theta)\partial_z.
\end{align*}
It is easy to verify
\begin{align*}
-z\partial_y+y\partial_z&\;=\;\partial_{\theta},\\
y\partial_x+x\partial_y&\;=\;\psi\cos\theta\partial_{\psi}-\sin\theta\partial_{\theta},\\
z\partial_x+z\partial_y+(x-y)\partial_z&\;=\;\psi\sin\theta\partial_{\psi}-(1-\cos\theta)\partial_{\theta}.
\end{align*}
\end{proof}
For the convenience of  the subsequent discussion, we introduce the following notations in advance. Let $\Pi$ represent the elliptic integral of the third kind, defined as
    \begin{align*}
     \Pi(\varphi,\alpha^2,k)&=\int^{\varphi}_0\frac{d\theta}{(1-\alpha^2\sin^2\theta)\sqrt{1-k^2\sin^2\theta}}\\
     &=\int_0^{\sin\varphi}\frac{dt}{(1-\alpha^2t^2)\sqrt{(1-k^2t^2)(1-t^2)}}.
    \end{align*}
Additionally, the elliptic integral of the first kind and the complete elliptic integral of the first kind are respectively expressed as
\begin{align*}
  F(\varphi,k)= \Pi\left(\varphi, 0,k\right),\qquad K(k)=F\left(\frac{\pi}{2},k\right).
\end{align*}
\section{The evolutions of curves on light-cone}\label{sec-ev}
Let us next investigate the evolutionary processes of curves on $LC^*$. We study a time-dependent family of smooth, local embedded, curves on $LC^*$, evolving by the flow:
\begin{equation*}
  \frac{\partial \mathbf{r}}{\partial t}=U\mathbf{r}+V\mathbf{Y}+W\mathbf{T}.
\end{equation*}
Then $\langle\mathbf{r}, \mathbf{r}\rangle_L=0$ and $\langle\mathbf{r},\mathbf{Y}\rangle_L=1$ gives $V=0$, which simplifies the evolution equation to
\begin{equation*}
  \frac{\partial \mathbf{r}}{\partial t}=U\mathbf{r}+W\mathbf{T}.
\end{equation*}
According to $g^2=\langle\mathbf{r}_p,\mathbf{r}_p\rangle_L$ and \eq{FS-eq-light}, we find
\begin{align*}
  gg_t&=\;\langle\mathbf{r}_{pt},\mathbf{r}_p\rangle_L\\
      &=\;g^2\langle(U\mathbf{r}+WT)_s,\mathbf{T}\rangle_L\\
      &=\;g^2(U+W_s).
\end{align*}
From this, we obtain
\begin{equation}\label{light-gt}
  \frac{g_t}{g}=U+W_s.
\end{equation}
Furthermore, we have the following derivatives
\begin{align*}
  \mathbf{T}_t&=\;(U_s+Wk_g)\mathbf{r}-W\mathbf{Y},\\
  \mathbf{Y}_t&=\;-(W_s+Wk_g)\mathbf{T}-U\mathbf{Y}.
\end{align*}
Concurrently, we compute
\begin{align*}
  (k_g)_t=\left(\langle\mathbf{T}_s,\mathbf{Y}\rangle_L\right)_t=\langle\mathbf{T}_{st},\mathbf{Y}\rangle_L+\langle\mathbf{T}_s,\mathbf{Y}_t\rangle_L,
\end{align*}
which yields
\begin{equation}\label{light-kt}
  (k_g)_t=U_{ss}-2k_gU+W(k_g)_s.
\end{equation}
Using \eqref{light-gt} and \eqref{light-kt}, we establish the following lemma
\begin{lem}\label{lem-kf-c}
    On $LC^*$, a vector field $\mathbf{J}=U\mathbf{r}+WT$  is designated as a Killing vector field along $\mathbf{r}$ if and only if it satisfies the following conditions:
  \begin{eqnarray*}
  \begin{aligned}
    &U+W_s=0,\\
    &U_{ss}-2k_gU+W(k_g)_s=0.
  \end{aligned}
  \end{eqnarray*}
  \end{lem}
\subsection{Harnack inequality on light-cone}\label{subsec-hk}
In particular, for the geometric heat equation on light-cone
\begin{equation*}
  \frac{\partial \mathbf{r}}{\partial t}=k_g\mathbf{r},
\end{equation*}
we have
\begin{equation*}
\frac{g_t}{g}=k_g,\qquad (k_g)_t=(k_g)_{ss}-2k_g^2.
\end{equation*}

Now let $k=-k_g$, and we obtain
\begin{equation}\label{klight-evo}
\frac{g_t}{g}=-k,\qquad k_t=k_{ss}+2k^2.
\end{equation}
\begin{lem}
If $k$ satisfies \eq{klight-evo}, then
$\displaystyle k_{\min}(t)=\inf_s\{k(s,t)\}$ is a nondecreasing function.
\end{lem}
Let us here give the proof of Theorem \ref{thm-hk}.
\begin{proof}[Proof of Theorem \ref{thm-hk}]
Define the Harnack quantity
  \begin{equation*}
      Q:=\frac{k_{ss}}{k}-\frac{k_s^2}{k^2}+k.
  \end{equation*}
  Then
  \begin{equation*}
    k_{ss}=kQ+\frac{k_s^2}{k}-k^2,
  \end{equation*}
  and by \eqref{klight-evo} we have
  \begin{equation*}
    k_{t}=kQ+\frac{k_s^2}{k}+k^2.
  \end{equation*}
  Using the commutation relation $[\partial_t, \partial_s]=k\partial_s$, we have
  \begin{align*}
    k_{st}&=\;k_{ts}+kk_s=Q_sk+3Qk_s+\frac{k_s^3}{k^2}+kk_s,\\
    k_{sst}&=\;kQ_{ss}+4Q_sk_s+3kQ^2+(\frac{6k_s^2}{k}-k^2)Q+\frac{k_s^4}{k^3}-2k^3.
  \end{align*}
  Thus,
  \begin{align}\label{Q0-exp}
    Q_{t}&=\;\frac{k_{sst}}{k}-\frac{k_{ss}k_t}{k^2}-\frac{2k_{s}k_{st}}{k^2}+\frac{2k_s^2k_t}{k^3}+k_t\triangleq \sum\limits_{j=1}^5I_j,
  \end{align}
  where
  \begin{eqnarray*}
   \begin{aligned}
   I_1=&\frac{k_{sst}}{k}=Q_{ss}+\frac{4Q_sk_s}{k}+3Q^2+(\frac{6k_s^2}{k^2}-k)Q+\frac{k_s^4}{k^4}-2k^2,\\
   I_2=&-Q^2-\frac{2k_s^2}{k^2}Q-\frac{k_s^4}{k^4}+k^2,\\
   I_3=&-\frac{6k_s^2}{k^2}Q-\frac{2k_s}{k}Q_s-\frac{2k_s^4}{k^4}-\frac{2k_s^2}{k},\\
   I_4=&\frac{2k_s^2}{k^2}Q+\frac{k_s^4}{k^4}+\frac{2k_s^2}{k}\\
   I_5=&kQ+k^2+\frac{k_s^2}{k}.
  \end{aligned}
  \end{eqnarray*}
  Inserting these expressions into \eqref{Q0-exp}, we arrive at
  \begin{equation*}
     Q_t=Q_{ss}+\frac{2k_s}{k}Q_s+2Q^2+\frac{k_s^2}{k}\geq Q_{ss}+\frac{2k_s}{k}Q_s+2Q^2.
  \end{equation*}
  Let $q(t)=-\frac{1}{2(t-t_0)}$. Then $q(t)$ solves the associated ODE
  \begin{equation*}
       \frac{dq}{dt}=2q^2
  \end{equation*}
  with condition $\displaystyle \lim_{t\to t_0}q(t)=-\infty$ and so the maximum principle gives
  \begin{equation*}
      Q(s,t)\geq q(t), \qquad \forall t\in(t_0,T).
  \end{equation*}
\end{proof}
\subsection{A third order curvature flow}
If the curve remains unstretched during its motion, meaning that the distance between any two points on the curve (as measured along the curve itself) remains constant over time, then $s$ and $t$ can function as local coordinates on the light-cone. Utilizing equation \eq{light-gt}, we derive
\begin{equation*}
U+W_s=0,
\end{equation*}
which subsequently allows us to express equation \eq{light-kt} as
\begin{align*}
(k_g)_t=U_{ss}-2k_gU-(k_g)_s\partial^{-1}U.
\end{align*}
By choosing $U=(k_g)_s$, we obtain the KdV equation
\begin{align*}
(k_g)_t=(k_g)_{sss}-3k_g(k_g)_s.
\end{align*}
Assume the curvature $k_g$ is periodic with the parameter $s$, that is, $k_g(s+T,t)=k_g(s,t)$, where $T$ is its period.
We can verify that
\begin{align*}
\frac{d}{dt}\int_0^Tds=0, \quad \frac{d}{dt}\int_0^Tk_gds=0, \quad \frac{d}{dt}\int_0^Tk_g^2ds=0.
\end{align*}

\section{The soliton solutions for the third order curvature flow}\label{sec-soliton}
Now let us consider the solitons for the flow \eqref{soflow}.
Solitons of the geometric flow represent self-similar solutions that maintain their shape under the evolution induced by the flow. To formulate the definition of a soliton, let $N^n$ be a $n$-dimensional Riemannian manifold with metric $g$, equipped with a Killing vector field $\mathbf{J}$ related to an isometry group $\varphi:N\times\R\to N$.
   The isometry group $\varphi$ characterizes transformations that preserve the metric $g$ and hence the geometric properties of the manifold. The relationship between $\mathbf{J}$ and $\varphi$ is given by the following differential equation and initial condition
  \begin{align*}
      \frac{d\varphi(x,t)}{dt}&=\;\mathbf{J}(\varphi(x,t)),\\
      \varphi(x,0)&=x.
  \end{align*}
  Here, $\frac{d\varphi(x,t)}{dt}$ denotes the time derivative of the point $\varphi(x,t)$ in the direction of the flow induced by the Killing vector field $\mathbf{J}$.
  The initial condition $\varphi(x,0)=x$ signifies that at time $t=0$, the isometry group leaves each point $x$ in $N^n$ unchanged.

  A curve $\mathbf{r}$ on $LC^*$ is a soliton of the geometric flow \eqref{lightflow} if, under the action of the isometry group $\varphi$ parameterized by the flow induced by $\mathbf{J}$, the curve evolves in such a way that its shape remains constant up to reparametrization. This property ensures that the soliton maintains its geometric characteristics throughout the evolution process.

 Hence, by using $U=(k_g)_s$ and $W=-k_g$, we have
 \begin{equation*}
    (k_g)_{sss}-3k_g(k_g)_s=0.
 \end{equation*}
 Through integration, we obtain
 \begin{equation*}
   (k_g)_{ss}-\frac{3}{2}k_g^2+\frac{\lambda}{2}=0, \quad \left((k_g)_s\right)^2-k_g^3+\lambda k_g+\mu=0,
 \end{equation*}
 where $\lambda$ and $\mu$ are the integration constants.
 \begin{rem} The trivial solution to \eqref{solitoneq} occurs when $k_g$ is constant. Through direct computation, it is evident that \eqref{solitoneq} possesses nontrivial periodic solutions if and only if the cubic equation $x^3-\lambda x-\mu=0$ has three distinct real roots. This condition is satisfied when $\displaystyle \lambda>3\left(\frac{\mu}{2}\right)^{2/3}$.
 \end{rem}
 Thus, the three solutions  of cubic equation $x^3-\lambda x-\mu=0$ are
 \begin{align*}
   x_1&\;=\;-\frac{2\sqrt{3\lambda}}{3}\cos\frac{\theta}{3},\\
   x_2&\;=\;\frac{2\sqrt{3\lambda}}{3}\cos\frac{\theta+\pi}{3},\\
   x_3&\;=\;\frac{2\sqrt{3\lambda}}{3}\cos\frac{\theta-\pi}{3},
 \end{align*}
 where $\displaystyle \theta=\arccos\left(-\frac{\mu}{2}\cdot\left(\frac{3}{\lambda}\right)^{3/2}\right)$ and $x_1<x_2<x_3$.
\begin{rem} If $\displaystyle \lambda>3\left(\frac{\mu}{2}\right)^{2/3}$, the periodic solutions for \eqref{solitoneq}  can be expressed in terms of the Jacobi elliptic sine function
\begin{equation*}
  k_g=x_1+(x_2-x_1)\mathrm{sn}^2\left(\frac{\sqrt{x_3-x_1}}{2}s,\;\sqrt{\frac{x_2-x_1}{x_3-x_1}}\right).
\end{equation*}
\end{rem}
\begin{rem} By  invoking \eqref{rem-k_g}, the task of identifying periodic solitons for the flow described by \eqref{soflow} is now equivalent to solving the equation
\begin{equation*}
   \psi_{ss}-\frac{\psi_{s}^2+1}{2\psi}-\left(x_1+(x_2-x_1)\mathrm{sn}^2\left(\frac{\sqrt{x_3-x_1}}{2}s,\;\sqrt{\frac{x_2-x_1}{x_3-x_1}}\right)\right)\psi=0.
\end{equation*}
In what follows, we employ the Killing vector field to derive analytic solutions for this second-order nonlinear differential equation.
\end{rem}
\subsection{Solitons rotating around a time-like axis}\label{subsec-time}
In fact, a periodic curve can be generated  by rotating the initial curve through one full period around an axis. The crucial step lies in determining the appropriate rotation matrix.
If the rotation axis is time-like, we can always assume, through a suitable rotation, that this axis aligns with the $x$-axis. Consequently, the rotation matrix can be expressed as
\begin{equation*}
\left(
  \begin{array}{ccc}
    1 & 0 & 0 \\
    0 & \cos\omega & -\sin\omega \\
    0 & \sin\omega & \cos\omega \\
  \end{array}
\right).
\end{equation*}
\begin{prop}\label{prop-mug0}
If the solitons rotate around the $x$-axis, then $\mu>0$ and $\psi$  is given by $\displaystyle\psi=-\frac{k_g}{\sqrt{\mu}}$.
\end{prop}
\begin{proof}
By Lemma \ref{lem-kf-c} and \eqref{solitoneq}, we discover
  \begin{align*}
    \mathbf{J}=(k_g)_s\mathbf{r}-k_g\mathbf{T}
  \end{align*}
  constitutes a Killing vector field along the curve $\mathbf{x}$.
  Now we can employ the coordinates $\mathbf{r}(\theta,\psi)=(\psi,\psi\cos\theta,\psi\sin\theta)$,
  so that according to Lemma \ref{lem=kf} its equator gives the only integral curve around $x$-axis of $\mathbf{J}:\partial_{\theta}=b\mathbf{J}$.
  Taking the inner product of both sides of this equation with $\mathbf{T}$ generates
 \begin{align*}
    \psi=-bk_g.
  \end{align*}
  Let $\bar{\mathbf{r}}$ denote the  integral curve of vector field $\displaystyle \bar{\mathbf{T}}=\frac{\mathbf{J}}{\left|\mathbf{J}\right|}$. Specifically,
  \begin{align*}
     \bar{\mathbf{T}}=\frac{(k_g)_s}{k_g}\mathbf{r}-\mathbf{T}.
  \end{align*}
  Therefore at the point where $\left(k_g\right)_s=0$, we need to find the derivative
  \begin{align*}
    \bar{\mathbf{T}}_{\bar{s}}\;=&\frac{ds}{d\bar{s}}\left(\frac{(k_g)_{ss}k_g-k_g^3}{k_g^2}\mathbf{r}+\mathbf{N}\right).
  \end{align*}
  At that given point, we have $\mathbf{r}=\bar{\mathbf{r}}$, which implies
  \begin{align*}
     \frac{ds}{d\bar{s}}=-1.
  \end{align*}
  Accordingly, by \eqref{solitoneq} and $(k_g)_s=0$,
  \begin{align*}
    \bar{k}_g\;=&-\frac{(k_g)_{ss}k_g-k_g^3}{k_g^2}=-\frac{k_g^3-\lambda k_g}{2k_g^2}=-\frac{\mu}{2k_g^2}.
  \end{align*}
  On the other hand, based on Proposition \ref{prop-plane}, the curvature of the plane curve that is perpendicular to the $x$-axis is given by $\displaystyle \bar{k}_g=-\frac{1}{2\psi^2}$.
  Consequently, we can derive
  \begin{align*}
    \mu>0,\quad \mathrm{and} \quad \frac{1}{b}=\sqrt{\mu}.
  \end{align*}
\end{proof}
 The following theorem provides the detailed expression for $\omega$.
 \begin{thm}\label{thm-ang}
 If $\mu>0$, by a rotation, then the progression angle  $\Lambda^{\Theta}$ in one period of the curvature can be expressed as follows
 \begin{equation}\label{ang}
     \Lambda^{\Theta}=-2\sqrt{\mu}\int^{x_2}_{x_1}\frac{1}{x\sqrt{x^3-\lambda x-\mu}}dx.
 \end{equation}
 \end{thm}
 \begin{proof}
 It is straightforward to observe that in one period, the arc length $T$ is given by
  \begin{align*}
    T=2\int^{x_2}_{x_1}\frac{1}{\sqrt{x^3-\lambda x-\mu}}dx=\frac{4}{\sqrt{x_3-x_1}}K\left(\sqrt{\frac{x_2-x_1}{x_3-x_1}}\right),
  \end{align*}
  where $K$ denotes the complete elliptic integral of the first kind.
  By Proposition \ref{prop-mug0}, $\displaystyle \frac{d\theta}{ds}=\frac{1}{\psi}$ in \eqref{rem-k_g}, and \eqref{solitoneq}, we have
  \begin{align}\label{or-A}
  \Lambda^{\Theta}=-\sqrt{\mu}\int^T_0\frac{1}{k_g}ds=-2\sqrt{\mu}\int^{x_2}_{x_1}\frac{1}{x\sqrt{x^3-\lambda x-\mu}}dx,
  \end{align}
  which achieves the anticipated  outcome.
 \end{proof}
 \begin{rem} If $\mu>0$, by selecting the initial conditions $\displaystyle \psi(0)=\frac{-x_1}{\sqrt{\mu}}$ and $\psi_s(0)=0$, we obtain the solution to equation \eqref{secod-ode} as follows $$\psi=\frac{x_1+(x_2-x_1)\mathrm{sn}^2\left(\frac{\sqrt{x_3-x_1}}{2}s,\;\sqrt{\frac{x_2-x_1}{x_3-x_1}}\right)}{\sqrt{\mu}}.$$
 \end{rem}
 \begin{rem}
 If $\mu\leq0$, then Proposition \ref{prop-mug0}  implies that the curve cannot be rotated around the time-like axis. According to the rotation matrix \eqref{RM}, the curves rotation round the space-like or light-like axis
 can not result in the formation of closed curves.  It is only when $\mu>0$ that the solitons are capable of forming closed curves.
 \end{rem}
 Based on \cite{gr}, we obtain an elegant expression for \eqref{ang}, that is,
   \begin{align*}
       \Lambda^{\Theta}=\frac{4\sqrt{\mu}}{-x_1\sqrt{x_3-x_1}}\Pi\left(\frac{\pi}{2},\frac{x_2-x_1}{-x_1},\sqrt{\frac{x_2-x_1}{x_3-x_1}}\right).
    \end{align*}
\begin{prop}\label{prop-lim} Assume $\mu>0$. Then,
\begin{itemize}
  \item[(1)] as $\displaystyle\lambda\to3\left(\frac{\mu}{2}\right)^{2/3}$, $\displaystyle \Lambda^{\Theta}\to 2\sqrt{2/3}\pi$;
  \item[(2)] as $\displaystyle\lambda\to+\infty$, $\displaystyle \Lambda^{\Theta}\to 2\pi$.
\end{itemize}
\end{prop}
\begin{proof} For $\displaystyle\lambda\to3\left(\frac{\mu}{2}\right)^{2/3}$, we have
\begin{equation*}
  x_1\to-\frac{\sqrt{3\lambda}}{3},\quad x_2\to-\frac{\sqrt{3\lambda}}{3}, \quad x_3\to\frac{2\sqrt{3\lambda}}{3}.
\end{equation*}
Through $\displaystyle \Pi\left(\frac{\pi}{2},0,0\right)=\frac{\pi}{2}$, we get $\displaystyle \Lambda^{\Theta}\to 2\sqrt{2/3}\pi$.

For $\displaystyle\lambda\to+\infty$, we have
\begin{equation*}
\theta\to\frac{\pi}{2},\qquad x_1\to-\infty,\qquad x_3\to+\infty.
\end{equation*}
Since $x_1x_2x_3=\mu$ and $x_1+x_2+x_3=0$, it is easy to find
\begin{equation*}
x_2\to 0,\qquad  \frac{x_1}{x_3}=-1,\qquad x_1^2x_2\to-\mu, \qquad \alpha^2\to1, \qquad k^2\to\frac{1}{2},
\end{equation*}
where $\displaystyle \alpha^2=\frac{x_2-x_1}{-x_1},\quad k^2=\frac{x_2-x_1}{x_3-x_1}$.
According to \cite{bf}, we obtain the expansion
 \begin{align*}
   \int^{\frac{\pi}{2}}_0\frac{d\theta}{(1-\alpha^2\sin^2\theta)\sqrt{1-k^2\sin^2\theta}}\;=&\;\int^{\frac{\pi}{2}}_0\frac{d\theta}{\sqrt{1-k^2\sin^2\theta}}\\
   &\quad\qquad+\frac{1}{1-k^2}\int^{\frac{\pi}{2}}_0\sqrt{1-k^2\sin^2\theta}d\theta\\
   &\quad\qquad+\frac{(2-k^2(1+\alpha^2))\pi}{4(1-k^2)\sqrt{1-k^2}\sqrt{1-\alpha^2}}\\
   &\quad\qquad +O(1-\alpha^2),
 \end{align*}
 where $\alpha^2\neq1$ and $O(1-\alpha^2)$ signifies that the remaining terms are bounded by $(1-\alpha^2)C$ for some constant $C$.
 One can verify that
 \begin{align*}
 \lim_{k^2\to\frac{1}{2}}\int^{\frac{\pi}{2}}_0\frac{d\theta}{\sqrt{1-k^2\sin^2\theta}}=0, \quad
 \lim_{k^2\to\frac{1}{2}}\int^{\frac{\pi}{2}}_0\sqrt{1-k^2\sin^2\theta}d\theta=0.
 \end{align*}
  Direct computation shows
 \begin{equation*}
   \lim_{\lambda\to+\infty}\frac{(2-k^2(1+\alpha^2))\pi}{4(1-k^2)\sqrt{1-k^2}\sqrt{1-\alpha^2}}\times\frac{4\sqrt{\mu}}{-x_1\sqrt{x_3-x_1}}=2\pi.
 \end{equation*}
 Hence, we complete the proof.
\end{proof}
Now let us consider the monotonicity  of $\Lambda^{\Theta}$.
\begin{rem}
If $\mu>0$, we can make $\mu=2$ in \eqref{solitoneq}. In fact, one can check that this amounts to replacing
	 $k_g,\; s$ by $\displaystyle \left(\frac{\mu}{2}\right)^{1/3}\bar{k}_g,\; \left(\frac{\mu}{2}\right)^{-2/3}p$, respectively.
In the following, we only consider the case $\mu=2$.
\end{rem}
By series expansions in \cite{bf}, if $k^2<1$ and $ k^2< \alpha^2$, it follows
\begin{equation}
	\Pi(v,k)=\Pi\left(\frac{\pi}{2},\alpha^2,k\right)=\sum_{m=0}^{\infty} c_{m}k^{2m},\label{vk}
	\end{equation}
where
\begin{align*}
  c_{0}&\;=\;\frac{\pi}{2\sqrt{1-\alpha^2}},\\
  c_{1}&\;=\;\frac{\pi}{4\alpha^2}\left(\frac{1}{\sqrt{1-\alpha^2}}-1\right),\\
  c_{2}&\;=\;\frac{3\pi}{32\alpha^4}\left(\frac{2}{\sqrt{1-\alpha^2}}-2-\alpha^2\right),\\
  c_{3}&\;=\;\frac{5\pi}{256\alpha^6}\left(\frac{8}{\sqrt{1-\alpha^2}}-8-4\alpha^2-3\alpha^4\right),
\end{align*}
\begin{align*}
	2(m+1)\alpha^2 c_{m+1} = \frac{\pi}{2(2m-1)} \genfrac{(}{)}{0pt}{}{-\frac{1}{2}}{m}^{2} + (1-2m) c_{m-1} + (2m+1+2mv)c_{m},
\end{align*}
and $\displaystyle \genfrac{(}{)}{0pt}{}{-\frac{1}{2}}{m}$ is the binomial coefficient.
Let $\displaystyle x=\left(3/\lambda\right)^{3/2}$, and then
\begin{align*}
  \frac{4\sqrt{\mu}}{-x_1\sqrt{x_3-x_1}}c_{0}&\;=\;\sqrt{2}\pi+\frac{\sqrt{6}}{18}\pi x+O(x^2),\\
  \frac{4\sqrt{\mu}}{-x_1\sqrt{x_3-x_1}}c_{1}k^2&\;=\;\frac{\sqrt{2}}{4}\pi-\frac{\sqrt[4]{3}}{6}\pi x^{1/2}-\frac{\sqrt{6}}{72}\pi x+O(x^2),\\
  \frac{4\sqrt{\mu}}{-x_1\sqrt{x_3-x_1}}c_{2}k^4&\;=\;\frac{3\sqrt{2}}{32}\pi-\frac{3\sqrt[4]{3}}{32}\pi x^{1/2}-\frac{\sqrt{6}}{64}\pi x+\frac{5\sqrt[4]{27}}{288}\pi x^{3/2}+O(x^2),\\
  \frac{4\sqrt{\mu}}{-x_1\sqrt{x_3-x_1}}c_{3}k^6&\;=\;\frac{5\sqrt{2}}{128}\pi-\frac{25\sqrt[4]{3}}{512}\pi x^{1/2}-\frac{25\sqrt{6}}{2304}\pi x\\
                    &\qquad+\frac{125\sqrt[4]{27}}{6912}\pi x^{3/2}+O(x^2),\\
  \frac{4\sqrt{\mu}}{-x_1\sqrt{x_3-x_1}}c_{4}k^8&\;=\;\frac{35\sqrt{2}}{2048}\pi-\frac{1225\sqrt[4]{3}}{49152}\pi x^{1/2}-\frac{245\sqrt{6}}{36864}\pi x\\
                    &\qquad+\frac{6125\sqrt[4]{27}}{442368}\pi x^{3/2}+O(x^2),\\
  \frac{4\sqrt{\mu}}{-x_1\sqrt{x_3-x_1}}c_{5}k^{10}&\;=\;\frac{63\sqrt{2}}{8192}\pi-\frac{6615\sqrt[4]{3}}{524288}\pi x^{1/2}-\frac{63\sqrt{6}}{16384}\pi x\\
                    &\qquad+\frac{1225\sqrt[4]{27}}{131072}\pi x^{3/2}+O(x^2),\\
  \frac{4\sqrt{\mu}}{-x_1\sqrt{x_3-x_1}}c_{6}k^{12}&\;=\;\frac{231\sqrt{2}}{65536}\pi-\frac{53361\sqrt[4]{3}}{8388608}\pi x^{1/2}-\frac{847\sqrt{6}}{393216}\pi x\\
                    &\qquad+\frac{148225\sqrt[4]{27}}{25165824}\pi x^{3/2}+O(x^2),\\
  \frac{4\sqrt{\mu}}{-x_1\sqrt{x_3-x_1}}c_{7}k^{14}&\;=\;\frac{32947\sqrt{2}}{201326592}\pi-\frac{429429\sqrt[4]{3}}{134217728}\pi x^{1/2}-\frac{1859\sqrt{6}}{1572864}\pi x\\
                   &\qquad+\frac{715715\sqrt[4]{27}}{201326592}\pi x^{3/2}+O(x^2).
\end{align*}
Finally, $\Lambda^{\Theta}$ can be expressed as
\begin{eqnarray}\label{app-r}
\begin{aligned}
  \Lambda^{\Theta}&\;=\;\frac{370345\sqrt{2}}{262144}\pi-\frac{47827845\sqrt[4]{3}}{134217728}\pi x^{1/2}-\frac{715\sqrt{6}}{524288}\pi x\\
                    &\qquad+\frac{123361315\sqrt[4]{27}}{1811939328}\pi x^{3/2}+O(x^2).
\end{aligned}
\end{eqnarray}
By the first four parts of this equation, we can see that $\Lambda^{\Theta}(\lambda)$ is monotonically increasing when $\displaystyle\lambda>3\left(\frac{\mu}{2}\right)^{2/3}$. By Proposition \ref{prop-lim}, we have $\Lambda^{\Theta}(\lambda)$  converges to $2\pi$ as $\lambda\to+\infty$, and converges to $\displaystyle2\sqrt{2/3}\pi$ as $\lambda\to3\left(\frac{\mu}{2}\right)^{2/3}$, which clearly shows $\displaystyle 2\sqrt{2/3}\pi<\Lambda^{\Theta}(\lambda)<2\pi$.
\begin{rem}
In Fig. \ref{Fig:compare},  the curves depicting $\Lambda^{\Theta}(\lambda)$ are precisely rendered using two distinct methods: the numerical integral approach defined in \eqref{ang}, and the approximate series expansions \eqref{app-r} with respect to $\lambda$. Notably, these two representations exhibit remarkable congruence, with minimal deviations among them, essentially indicating a near-perfect overlap.
 \end{rem}
\begin{figure}[htpb]
\centering
\setlength{\abovecaptionskip}{0.cm}
\setlength{\abovecaptionskip}{0.cm}
	\includegraphics[width=0.7\textwidth]{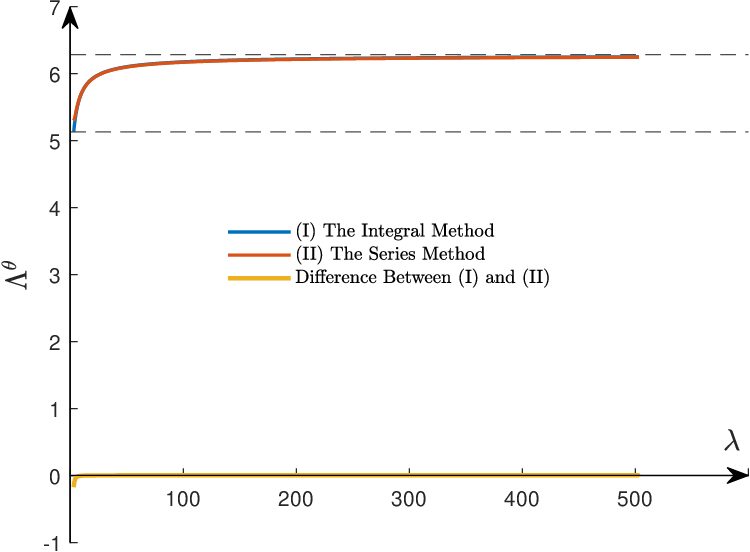}
	\caption{The comparison of different method to calculate $\Lambda^{\Theta}$.}
\label{Fig:compare}
\end{figure}

\subsection{Solitons rotating around a light-like axis}
If the rotation axis is light-like, it is always possible, through an appropriate rotation, to assume it aligns with the vector $(1,1,0)$. Consequently, the corresponding rotation matrix can be formulated as
\begin{equation}\label{lightmatrix}
\left(
  \begin{array}{ccc}
    1+\frac{\omega^2}{2} & -\frac{\omega^2}{2} & \omega \\
    \frac{\omega^2}{2} & 1-\frac{\omega^2}{2} & \omega\\
    \omega & -\omega & 1 \\
  \end{array}
\right).
\end{equation}
\begin{prop}\label{prop-mu0}
When solitons rotate around the light-like axis  $(1,1,0)$, it follows that $\mu=0$ and $\psi$  is given by
\begin{align*}
(1-\cos\theta)\psi=-bk_g,
\end{align*}
 where $b$ denotes an arbitrary positive constant.
\end{prop}
\begin{proof}
By Lemma \ref{lem-kf-c} and \eqref{solitoneq}, we discover
  \begin{align*}
    \mathbf{J}=(k_g)_s\mathbf{r}-k_g\mathbf{T}
  \end{align*}
  constitutes a Killing vector field along the curve $\mathbf{x}$.
  Now we can employ the coordinates $\mathbf{x}(\theta,\psi)=(\psi,\psi\cos\theta,\psi\sin\theta)$,
  so that according to Lemma \ref{lem=kf} its equator gives the only integral curve around light-like axis $(1,1,0)$ of $\mathbf{J}:\psi\sin\theta\partial_{\psi}-(1-\cos\theta)\partial_{\theta}=-b\mathbf{J}$.
  Taking the inner product of both sides of this equation with $\mathbf{T}$ generates
 \begin{align*}
    (1-\cos\theta)\psi=-bk_g.
  \end{align*}
  Let $\bar{\mathbf{r}}$ denote the  integral curve of vector field $\displaystyle \bar{\mathbf{T}}=\frac{\mathbf{J}}{\left|\mathbf{J}\right|}$.
  Similar calculations as in Proposition \ref{prop-mug0}, we have
  \begin{align*}
    \bar{k}_g\;=&-\frac{(k_g)_{ss}k_g-k_g^3}{k_g^2}=-\frac{k_g^3-\lambda k_g}{2k_g^2}=-\frac{\mu}{2k_g^2}.
  \end{align*}
  On the other hand, based on Proposition \ref{prop-plane}, the curvature of the plane curve that is perpendicular to the axis $(1,1,0)$ is given by $\displaystyle \bar{k}_g=0$.
  Consequently, we can derive  $\mu=0$.
\end{proof}
Then we have
\begin{align*}
  \sin\theta+(1-\cos\theta)\psi_s&\;=\;-b(k_g)_s,\\
  \frac{\cos\theta}{\psi}+\frac{\psi_s}{\psi}\sin\theta+(1-\cos\theta)\psi_{ss}&\;=\;-b(k_g)_{ss}.
\end{align*}
 \begin{rem}\label{rem-02pi}
 In the first period, the maximum value of $k_g$ is zero, which we assume is achieved at the point where $\theta=2\pi$.
 Subsequently, all instances where $k_g=0$ arise at $\theta=2n\pi$, where $n$ is any positive integer.
 Notably, at these points, $\psi$ takes on a constant value of $\displaystyle \frac{2}{b\lambda}$. Consequently, these points manifest as a same point on the light-cone.
\end{rem}
It is straightforward to observe that in one period, the arc length $T$ is given by
  \begin{align*}
    T=2\int^{0}_{-\sqrt{\lambda}}\frac{1}{\sqrt{x^3-\lambda x}}dx=\frac{4}{\sqrt{2\sqrt{\lambda}}}K\left(\sqrt{\frac{1}{2}}\right),
  \end{align*}
  where $K$ denotes the complete elliptic integral of the first kind.

Now, let us consider the initial conditions $\psi(0)=-x_1$ and $\psi_s(0)=0$. These conditions imply that $\theta=\pi$ at $s=0$ and $b=2$. Furthermore,
\begin{align*}
 \psi\sin^2\left(\frac{1}{2}\int_0^s\frac{dt}{\psi(t)}+\frac{\pi}{2}\right)=-k_g.
\end{align*}
Then we have
\begin{align*}
 \psi=-\left(1+\frac{1}{4}\left(\int_0^s\frac{dt}{k_g(t)}\right)^2\right)k_g(s).
\end{align*}
Based on \cite{gr}, according to \eqref{solitoneq}, we have when $\displaystyle s<\frac{T}{2}$,
\begin{align*}
 \psi\;=\;&-\left(1+\frac{1}{4}\left(\int_{x_1}^{k_g(s)}\frac{dx}{x\sqrt{x^3-\lambda x}}\right)^2\right)k_g(s)\\
     \;=\;&-\left(1+\frac{1}{2\lambda\sqrt{\lambda}}\left(\Pi\left(\arcsin\sqrt{\frac{k_g(s)+\sqrt{\lambda}}{\sqrt{\lambda}}},\;1,\;\sqrt{\frac{1}{2}}\right)\right)^2\right)k_g(s)\\
     \;=\;&-\left(1+\frac{1}{2\lambda\sqrt{\lambda}}\left(\Pi\left(\arcsin\left|\mathrm{sn}\left(\frac{\sqrt[4]{\lambda}}{\sqrt{2}}s,\;\sqrt{\frac{1}{2}}\right)\right|,\;1,\;\sqrt{\frac{1}{2}}\right)\right)^2\right)\times\\
          &\qquad\sqrt{\lambda}\left(-1+\mathrm{sn}^2\left(\frac{\sqrt[4]{\lambda}}{\sqrt{2}}s,\;\sqrt{\frac{1}{2}}\right)\right).
\end{align*}
\begin{rem} When $\mu=0$,  by specifying the initial conditions $\displaystyle \psi(0)=-x_1$ and $\psi_s(0)=0$, we find that the solution to equation \eqref{secod-ode} is given by
 \begin{equation*}
   \psi\sin^2\left(\frac{1}{2}\int_0^s\frac{dt}{\psi(t)}+\frac{\pi}{2}\right)=x_1\left(-1+\mathrm{sn}^2\left(\frac{\sqrt{-x_1}}{\sqrt{2}}s,\;\sqrt{\frac{1}{2}}\right)\right).
 \end{equation*}
 Alternatively, if $\displaystyle s<\frac{T}{2}$, this can be expressed as
 \begin{align*}
 \psi\;=\;&-\left(1-\frac{1}{2x_1^3}\left(\Pi\left(\phi,\;1,\;\sqrt{\frac{1}{2}}\right)\right)^2\right)\times\\
          &\qquad\qquad\qquad\qquad\qquad\qquad x_1\left(1-\mathrm{sn}^2\left(\frac{\sqrt{-x_1}}{\sqrt{2}}s,\;\sqrt{\frac{1}{2}}\right)\right),
 \end{align*}
 where $\displaystyle \phi=\mathrm{am}\left(\frac{\sqrt{-x_1}}{\sqrt{2}}s, \;\sqrt{\frac{1}{2}}\right)$, and $\mathrm{am}$ denotes  the jacobi amplitude function.
\end{rem}
\begin{prop}\label{prop-b-lignt}
If the solitons rotate around the light-like axis $(1,1,0)$, then $\psi$ is unbounded and its  infimum  is $0$.
\end{prop}
\begin{proof}
 By applying the rotation matrix given in \eqref{lightmatrix}, after $n$ cycles, points $(\psi,\psi\cos\theta,\psi\sin\theta)$  in the first cycle are rotated to the points
 \begin{align*}
 \Bigg(1+\frac{(n\omega)^2}{2}&-\frac{(n\omega)^2}{2}\cos\theta+n\omega\sin\theta, \\
 &\frac{(n\omega)^2}{2}+\left(1-\frac{(n\omega)^2}{2}\right)\cos\theta+n\omega\sin\theta,
  n\omega-n\omega\cos\theta+\sin\theta \Bigg)\psi.
 \end{align*}
For $\theta_0=\pi$, $\psi_n=\left(1+(n\omega)^2\right)\psi_0$, where $\psi_0$ is the value for $\psi$ at $\theta_0$ in the first period. $\psi_n$ tends to $+\infty$ as $n\to+\infty$, which demonstrates that $\psi$ is unbounded.

Next, we consider the function
\begin{equation*}
  f(\theta)=1+\frac{(n\omega)^2}{2}-\frac{(n\omega)^2}{2}\cos\theta+n\omega\sin\theta.
\end{equation*}
Its derivative is given by
\begin{equation*}
  f'(\theta)=\frac{(n\omega)^2}{2}\sin\theta+n\omega\cos\theta.
\end{equation*}
Setting $f'(\bar{\theta}_n)=0$ leads to
\begin{align*}
\sin{\bar{\theta}_n}=-\frac{2}{\sqrt{4+(n\omega)^2}},\quad \cos{\bar{\theta}_n}=\frac{n\omega}{\sqrt{4+(n\omega)^2}}.
\end{align*}
Furthermore, the second derivative of $f$ at $\bar{\theta}_n$ is
\begin{align*}
  f''(\bar{\theta}_n)=\frac{(n\omega)^2}{2}\cos{\bar{\theta}_n}-n\omega\sin{\bar{\theta}_n}=\frac{n\omega\sqrt{4+(n\omega)^2}}{2}>0,
\end{align*}
indicating that $\bar{\theta}_n$ is a minimum point of $f$.
At this minimum point, the value of $f$ is:
\begin{align*}
  f(\bar{\theta}_n)=1-\frac{2}{1+\sqrt{1+4/(n\omega)^2}}.
\end{align*}
As $n\to +\infty$, $f(\bar{\theta}_n)$ tends to $0$.
Therefore, the infimum of $\psi$ (considering its dependence on $f(\theta)$ through the rotation) is $0$.
\end{proof}
If $\theta(0)=\pi$, according to the preceding proof, we derive the following expressions
\begin{align*}
  \sin\theta(nT)=\frac{2n\omega}{1+(n\omega)^2},\quad \cos\theta(nT)=\frac{(n\omega)^2-1}{1+(n\omega)^2}.
\end{align*}
From these expressions, we deduce the following corollary.
\begin{cor}\label{cor-light}
When solitons rotate around the light-like axis $(1,1,0)$ with the initial condition $\theta(0)=\pi$, it holds that  $2n\pi<\theta(nT)<(2n+1)\pi$.
The angle $\theta(nT)-2n\pi$ decreases monotonically as $n$ increases, and as the integer $n$ approaches positive infinity, $\theta(nT)-2n\pi$ converges to zero.
\end{cor}

\begin{rem} Given the initial conditions  $\psi(0)=-x_1$ and $\psi_s(0)=0$, our numerical calculations (depicted in the left-hand side of  Fig. \ref{Fig:0lambda}) reveal that
$\theta(T,\lambda)$ exhibits monotonic increasing behavior with respect to $\lambda$. Furthermore, we observe that  $\displaystyle \lim_{\lambda\to0}\theta(T,\lambda)=2\pi$, whereas $\displaystyle \lim_{\lambda\to+\infty}\theta(T,\lambda)=3\pi$.
\end{rem}

\begin{figure}[htpb]
\centering
\setlength{\abovecaptionskip}{0.cm}
\setlength{\abovecaptionskip}{0.cm}
	\includegraphics[width=0.45\textwidth]{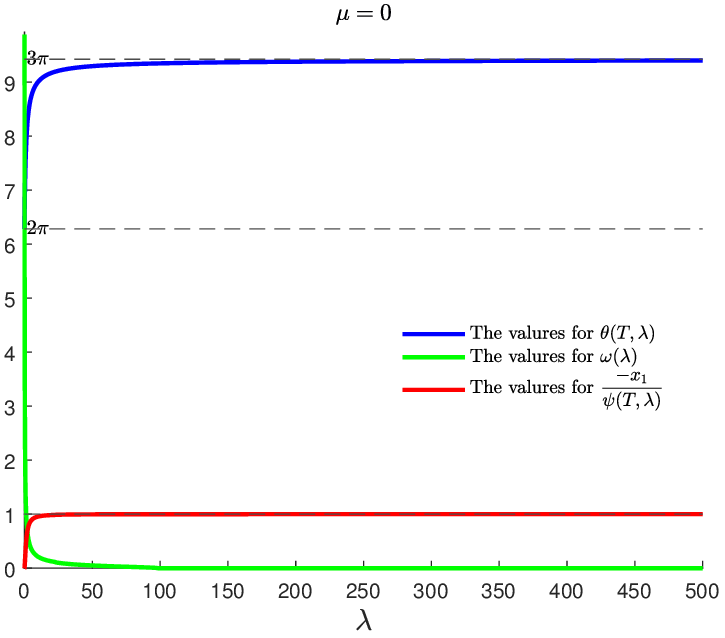}\qquad
    \includegraphics[width=0.45\textwidth]{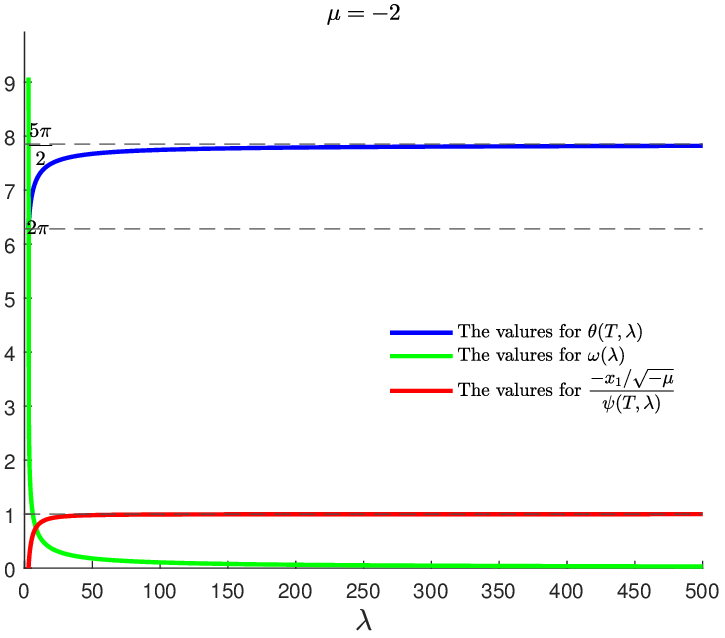}
	\caption{Monotonicity \& asymptotics at $s=T$ with varying $\lambda$.}
\label{Fig:0lambda}
\end{figure}

\subsection{Solitons rotating around a space-like axis}
If the rotation axis is space-like, we can always assume, through a suitable rotation, that this axis aligns with the $z$-axis. Consequently, the rotation matrix can be expressed as
\begin{equation}\label{spacematrix}
\left(
  \begin{array}{ccc}
    \cosh\omega & \sinh\omega & 0 \\
    \sinh\omega & \cosh\omega & 0\\
    0 & 0 & 1 \\
  \end{array}
\right).
\end{equation}
By employing methodologies akin to those utilized in Proposition \ref{prop-mug0} and Proposition \ref{prop-mu0}, we can formulate the subsequent proposition.
\begin{prop}
If the solitons rotate around the $z$-axis, it necessitates that $\mu<0$ and $\psi$  is expressed as
\begin{align*}
\psi\sin\theta=\frac{-k_g}{\sqrt{-\mu}}.
\end{align*}
\end{prop}
It is straightforward to observe that in one period, the arc length $T$ is given by
  \begin{align*}
    T=2\int^{x_2}_{x_1}\frac{1}{\sqrt{x^3-\lambda x-\mu}}dx=\frac{4}{\sqrt{x_3-x_1}}K\left(\sqrt{\frac{x_2-x_1}{x_3-x_1}}\right),
  \end{align*}
  where $K$ denotes the complete elliptic integral of the first kind.

 In  one period,  $k_g$ demonstrates the presence of two zero points, occurring at $\theta=\pi$ and $\theta=2\pi$,  respectively,  due to the condition that $\psi>0$.
 According to \eqref{solitoneq}, we have
 \begin{align}\label{mumd}
   \cos\theta+\psi_s\sin\theta=\frac{-(k_g)_s}{\sqrt{-\mu}}.
 \end{align}
 \eqref{solitoneq} yields $(k_g)_s=\pm\sqrt{-\mu}$ at $\theta=\pi$ and $\theta=2\pi$.

 Then, we choose the initial condition  $\displaystyle\psi(0)=\frac{-x_1}{\sqrt{-\mu}}$, which yields  $\displaystyle\theta=\frac{\pi}{2}$ at $s=0$.
 From \eqref{mumd}, we deduce that  $\psi_s(0)=0$.
 Given these conditions, we obtain
 \begin{equation*}
   \psi\sin\left(\int_0^s\frac{dt}{\psi(t)}+\frac{\pi}{2}\right)=\frac{-k_g}{\sqrt{-\mu}}.
 \end{equation*}
 Since  $\theta=\pi$ and $\theta=2\pi$ arise in the initial period, the range for $\theta$ is $\pi/2\leq\theta<3\pi$, and at the endpoint, $\theta>2\pi$.
 Once again, utilizing $\displaystyle \frac{ds}{d\theta}=\psi$ in \eqref{rem-k_g}, we find when $\displaystyle s<\frac{2}{\sqrt{x_3-x_1}}F\left(\arcsin\sqrt{\frac{-x_1}{x_2-x_1}},\sqrt{\frac{x_2-x_1}{x_3-x_1}}\right)$,
 \begin{equation*}
   \psi=\cosh\left(\int^s_0\frac{\sqrt{-\mu}}{k_g(t)}dt\right)\frac{-k_g(s)}{\sqrt{-\mu}}.
 \end{equation*}
 Based on \cite{gr}, according to \eqref{solitoneq}, we have
 \begin{align*}
   \psi\;=\;&\cosh\left(\sqrt{-\mu}\int^{k_g(s)}_{x_1}\frac{1}{x\sqrt{x^3-\lambda x-\mu}}dx\right)\frac{-k_g(s)}{\sqrt{-\mu}}\\
       \;=\;&\cosh\left(\frac{2\sqrt{-\mu}}{x_1\sqrt{x_3-x_1}}\Pi\left(\arcsin\sqrt{\frac{k_g-x_1}{x_2-x_1}},\frac{x_2-x_1}{-x_1},\sqrt{\frac{x_2-x_1}{x_3-x_1}}\right)\right)\frac{-k_g(s)}{\sqrt{-\mu}},
 \end{align*}
 when $\displaystyle s<\frac{2}{\sqrt{x_3-x_1}}F\left(\arcsin\sqrt{\frac{-x_1}{x_2-x_1}},\sqrt{\frac{x_2-x_1}{x_3-x_1}}\right)$.
 \begin{rem} When $\mu<0$,  by specifying the initial conditions $\displaystyle \psi(0)=\frac{-x_1}{\sqrt{-\mu}}$ and $\psi_s(0)=0$, we find that the solution to equation \eqref{secod-ode} is given by
 \begin{equation*}
   \psi\sin\left(\int_0^s\frac{dt}{\psi(t)}+\frac{\pi}{2}\right)=-\frac{x_1+(x_2-x_1)\mathrm{sn}^2\left(\frac{\sqrt{x_3-x_1}}{2}s,\;\sqrt{\frac{x_2-x_1}{x_3-x_1}}\right)}{\sqrt{-\mu}}.
 \end{equation*}
 Alternatively, when $\displaystyle s<\frac{2}{\sqrt{x_3-x_1}}F\left(\arcsin\sqrt{\frac{-x_1}{x_2-x_1}},\sqrt{\frac{x_2-x_1}{x_3-x_1}}\right)$,
 this can be expressed as
 \begin{align*}
 \psi\;=\;&-\frac{1}{\sqrt{-\mu}}\cosh\left(\frac{2\sqrt{-\mu}}{x_1\sqrt{x_3-x_1}}\Pi\left(\phi,\frac{x_2-x_1}{-x_1},\sqrt{\frac{x_2-x_1}{x_3-x_1}}\right)\right)\times\\
          &\left(x_1+(x_2-x_1)\mathrm{sn}^2\left(\frac{\sqrt{x_3-x_1}}{2}s,\;\sqrt{\frac{x_2-x_1}{x_3-x_1}}\right)\right),
 \end{align*}
 where $\displaystyle \phi=\mathrm{am}\left(\frac{\sqrt{x_3-x_1}}{2}s, \;\sqrt{\frac{x_2-x_1}{x_3-x_1}}\right)$, and $\mathrm{am}$ denotes  the jacobi amplitude function.
\end{rem}
\begin{prop}\label{prop-b-space}
If the solitons rotate around the $z$-axis,, then $\psi$ is unbounded and its  infimum  is $0$.
\end{prop}
\begin{proof}
 By applying the rotation matrix given in \eqref{spacematrix}, after $n$ cycles, points $(\psi,\psi\cos\theta,\psi\sin\theta)$  in the first cycle are rotated to the points
 \begin{align*}
 \Bigg(\cosh(n\omega)+\sinh(n\omega)\cos\theta,\sinh(n\omega)+\cosh(n\omega)\cos\theta, \sin\theta\Bigg)\psi.
 \end{align*}
 To analyze the behavior of $\psi$, we consider specific values of $\theta$.
For $\theta_0=\pi$, $\psi_n=\exp{\left(-n\omega\right)}\psi_0$, where $\psi_0$ is the value of $\psi$ at $\theta=\theta_0$ in the first period.  As $n\to+\infty$, $\psi_n\to 0$.
For $\theta_0=2\pi$, $\psi_n=\exp{\left(n\omega\right)}\psi_0$, where again $\psi_0$ is the value of $\psi$ at $\theta=\theta_0$ in the first period.  As $n\to+\infty$, $\psi_n\to +\infty$.
Thus, we have shown that $\psi$ is unbounded and its infimum is $0$, as desired.
\end{proof}

Given the initial condition $\displaystyle \theta(0)=\frac{\pi}{2}$, and following the preceding proof, we arrive at the subsequent expressions
\begin{align*}
  \sin\theta(nT)=\frac{1}{\cosh(n\omega)},\quad \cos\theta(nT)=\frac{\sinh(n\omega)}{\cosh(n\omega)}.
\end{align*}
From these derived expressions, we formulate the following corollary.
\begin{cor}\label{cor-space}
When solitons rotate around the $z$-axis with the initial condition $\displaystyle \theta(0)=\frac{\pi}{2}$, it holds that  $\displaystyle 2n\pi<\theta(nT)<\frac{(4n+1)\pi}{2}$.
The angle $\theta(nT)-2n\pi$ decreases monotonically as $n$ increases, and as the integer $n$ approaches positive infinity, $\theta(nT)-2n\pi$ converges to zero.
\end{cor}

\begin{rem} With the initial conditions  $\displaystyle\psi(0)=\frac{-x_1}{\sqrt{-\mu}}$ and $\psi_s(0)=0$, our numerical calculations (depicted in the right-hand side of  Fig. \ref{Fig:0lambda}) reveal that
$\theta(T,\lambda)$ exhibits a monotonically increasing behavior with respect to $\lambda$.
Additionally, we observe the following asymptotic behavior:  $\theta(T,\lambda)\to2\pi$ as $\displaystyle \lambda\to3\left(\mu/2\right)^{2/3}$, and $\theta(T,\lambda)\to\frac{5\pi}{2}$ as $\displaystyle \lambda\to+\infty$.
\end{rem}

\subsection{Proof of Theorem \ref{thm-ode}}\label{subsec-ode}
Based on previous insights, we have developed a method for solving the differential equation \eqref{secod-ode}.
\begin{proof}[Proof of Theorem \ref{thm-ode}]
  Through meticulous computation, we verify that
  \begin{align*}
     2ff_{ss}-f_s^2-2f^3=x_1x_2x_3.
  \end{align*}
  Following the guidance provided by previous hints, we postulate that the solution can be formulated as
  \begin{align*}
     \psi(s)=f(s)G(\theta),
  \end{align*}
  and $\theta$ is defined through the differential relationship $\displaystyle \frac{d\theta}{ds}=\frac{1}{\psi}$.
Consequently, this second-order differential equation can be rewritten as
\begin{align*}
  G_{\theta\theta}-\frac{3G_{\theta}^2}{2G}-\frac{G}{2}+\frac{G^3}{2}x_1x_2x_3=0.
\end{align*}
Utilizing Proposition \ref{prop-plane}, we are able to draw the desired conclusions.
\end{proof}

\noindent {\bf Data availability} No data availability statement is required, as no experimental data is involved.

\noindent{\bf Conflict of interest} The authors have no conflict of interest to declare that are relevant to the content of this article.


\end{document}